\pdfoutput=1
\documentclass[letterpaper,    
               english,        
               12pt,           
               twoside]        
               {amsart}        
\usepackage[english]{babel}                   
\usepackage[bookmarks,                        
            bookmarksopen=true,               
            bookmarksnumbered=true,           
            pdfauthor={Christof Puhle},       
            pdftitle={Almost contact metric   
                      5-manifolds and
                      connections with torsion},
            pdfkeywords={Almost contact metric structures,
                         connections with torsion},
            colorlinks,                       
            linkcolor=black,                  
            citecolor=black,                  
            urlcolor=black,                   
            plainpages=false]                 
            {hyperref}                        
\usepackage[nobysame]
           {amsrefs}                          
\usepackage{fourier}                          
\usepackage{bm}                               
\usepackage{color}                            
\advance\oddsidemargin by -0.93cm
\advance\evensidemargin by -0.92cm
\advance\topmargin by -0.55cm
\theoremstyle{plain}
\newtheorem{cor}{Corollary}[section]
\newtheorem{lem}{Lemma}[section]
\newtheorem{thm}{Theorem}[section]            
\newtheorem{prop}{Proposition}[section]
\theoremstyle{definition}

\newtheorem{rmk}{Remark}[section]

\newcommand{\be}[1][*]{\begin{equation#1}}
\newcommand{\ee}[1][*]{\end{equation#1}}
\newcommand{\ba}{\be\begin{aligned}}
\newcommand{\ea}{\end{aligned}\ee}
\newcommand{\barr}[1]{\begin{array}{#1}}
\newcommand{\earr}{\end{array}}
\newcommand{\bite}{\begin{itemize}}
\newcommand{\eite}{\end{itemize}}
\newcommand{\btab}[1]{\renewcommand{\arraystretch}{1.2}\begin{center}\begin{tabular}{#1}}
\newcommand{\etab}{\end{tabular}\end{center}\renewcommand{\arraystretch}{1.0}}
\newcommand{\bnum}{\begin{enumerate}}
\newcommand{\enum}{\end{enumerate}}
\newcommand{\bcen}{\begin{center}}
\newcommand{\ecen}{\end{center}}
\newcommand{\mca}[1]{\mathcal{#1}}
\newcommand{\mrm}[1]{\mathrm{#1}}
\newcommand{\mfr}[1]{\mathfrak{#1}}
\def\side#1{\ifvmode\leavevmode\fi\vadjust{\vbox to0pt{\vss
\hbox to 0pt{\hskip\hsize\hskip1em                     
\vbox{\hsize2cm\small\raggedright\pretolerance10000       
\noindent\textcolor{red}{#1}\hfill}\hss}\vbox to8pt{\vfil}\vss}}}
\newcommand{\x}{\ensuremath{\times}}
\newcommand{\op}{\ensuremath{\oplus}}
\newcommand{\ox}{\ensuremath{\otimes}}

\newcommand{\lb}{\ensuremath{\left}}
\newcommand{\rb}{\ensuremath{\right}}

\newcommand{\lan}{\ensuremath{\left\langle}}
\newcommand{\ran}{\ensuremath{\right\rangle}}
\newcommand{\hook}{\ensuremath{\lrcorner\,}}
\newcommand{\ra}{\ensuremath{\rightarrow}}

\newcommand{\setsep}{\ensuremath{\,\big|\,}}
\newcommand{\ex}{\ensuremath{\exists\,}}

\newcommand{\del}{\ensuremath{\partial}}

\newcommand{\R}{\ensuremath{\mathbb{R}}}
\newcommand{\C}{\ensuremath{\mathbb{C}}}

\newcommand{\W}{\ensuremath{\mathcal{W}}}

\newcommand{\T}{\ensuremath{\mathcal{T}}}
\newcommand{\A}{\ensuremath{\mathcal{A}}}

\newcommand{\vphi}{\ensuremath{\varphi}}    
\newcommand{\vrho}{\ensuremath{\varrho}}

\newcommand{\diag}{\ensuremath{\mathrm{diag}}}
\newcommand{\Id}{\ensuremath{\mathrm{Id}}}

\newcommand{\pr}{\ensuremath{\mathrm{pr}}}
\newcommand{\Ric}{\ensuremath{\mathrm{Ric}}}

\newcommand{\hol}{\ensuremath{\mathfrak{hol}}}

\newcommand{\GL}{\ensuremath{\mathrm{GL}}}

\newcommand{\un}{\ensuremath{\mathfrak{u}}}
\newcommand{\Un}{\ensuremath{\mathrm{U}}}
\newcommand{\su}{\ensuremath{\mathfrak{su}}}

\newcommand{\Orth}{\ensuremath{\mathrm{O}}}
\newcommand{\so}{\ensuremath{\mathfrak{so}}}
\newcommand{\SO}{\ensuremath{\mathrm{SO}}}

\newcommand{\m}{\ensuremath{\mathfrak{m}}}

\begin{document}
\thispagestyle{empty}
\date{\today}
\title{Almost contact metric 5-manifolds and connections with torsion}
\author{Christof Puhle}
\address{
Institut f\"ur Mathematik \newline\indent
Humboldt-Universit\"at zu Berlin\newline\indent
Unter den Linden 6\newline\indent
10099 Berlin, Germany}
\email{\noindent puhle@math.hu-berlin.de}
\urladdr{www.math.hu-berlin.de/~puhle}
\begin{abstract}
We study $5$-dimensional Riemannian manifolds that admit an almost contact metric structure. We classify these structures by their intrinsic torsion and review the literature in terms of this scheme. Moreover, we determine necessary and sufficient conditions for the existence of metric connections with vectorial, totally skew-symmetric or traceless cyclic torsion that are compatible with the almost contact metric structure. Finally, we examine explicit examples of almost contact metric $5$-manifolds from this perspective.
\end{abstract}
\maketitle
\setcounter{tocdepth}{1}
\bcen
\begin{minipage}{0.7\linewidth}
    \begin{small}
      \tableofcontents
    \end{small}
\end{minipage}
\ecen
\pagestyle{headings}
%
%
%
%
\section{Introduction}\noindent
In the 1920s, Cartan classified metric connections on Riemannian manifolds $\lb(M^n,g\rb)$ by the algebraic type of the corresponding torsion tensor (see \cite{Car25}). A central result is that for $n\geq3$, the space $\T$ of possible torsion tensors splits into the direct sum of three irreducible $\Orth\lb(n\rb)$-modules,
\be
\T=\T_1\op\T_2\op\T_3.
\ee
Consequently, there are three principal types of torsion: type $\T_1$ (vectorial torsion), type $\T_2$ (totally skew-symmetric torsion) and type $\T_3$ (traceless cyclic torsion). In contrast to the first two cases, metric connections with traceless cyclic torsion remain unexplored to this day (consult \cite{Agr06} for an overview).

An almost contact metric manifold is an odd-dimensional Riemannian manifold $\lb(M^{2k+1},g\rb)$ such that there exists a reduction of the structure group of orthonormal frames of the tangent bundle to $\Un\left(k\right)$ (see \cite{Gra59}). As shown in \cites{SH61,SH62}, an almost contact metric structure on $\lb(M^{2k+1},g\rb)$ can be equivalently defined by a triple $\lb(\xi,\eta,\vphi\rb)$ of tensor fields satisfying certain conditions (see \autoref{sec:3}).

The purpose of this paper is to investigate almost contact metric $5$-manifolds with regard to the existence of metric connections $\nabla^c$ with vectorial, totally skew-symmetric or traceless cyclic torsion that are compatible with the underlying almost contact metric structure, i.e.\
\ba
\nabla^c\xi&=0, &
\nabla^c\eta&=0, &
\nabla^c\vphi&=0.
\ea
We proceed as follows:

Firstly, we study the algebra related to the action of the group $\Un(2)$ (see \autoref{sec:2}). Amongst other things, we show that for an almost contact metric $5$-manifold the space $\T$ of possible torsion tensors splits into $15$ irreducible $\Un(2)$-modules (see corollary \ref{cor:1}),
\ba
\T_1&=\T_{1,1}\op\T_{1,2}, &
\T_2&=\T_{2,1}\op\ldots\op\T_{2,4}, &
\T_3&=\T_{3,1}\op\ldots\op\T_{3,9}.
\ea

Secondly, we follow the method of \cite{Fri03} and classify almost contact metric $5$-mani\-folds with respect to the algebraic type of the corresponding intrinsic torsion tensor $\Gamma$ (see \autoref{sec:3}). There are $10$ irreducible $\Un(2)$-modules $\W_1,\ldots,\W_{10}$ in the decomposition of the space of possible intrinsic torsion tensors (see proposition \ref{prop:1}):
\be
\Gamma\in\W_1\op\ldots\op\W_{10}.
\ee
Therefore, there exist $2^{10}=1024$ classes according to this approach. Obviously, most of them have never been studied. We introduce those carrying names and review them, in the light of our classification scheme, in section \ref{sec:3.2}. To mention just two examples, quasi-Sasaki manifolds (see \cite{Bla67}) correspond to the case $\Gamma\in\W_3\op\W_5$, and trans-Sasaki manifolds (see \cite{Oub85}) correspond to the class $\W_1\op\W_3$ (see theorems \ref{thm:2} and \ref{thm:3}). In section \ref{sec:3.1}, we relate our classification scheme to the work of Chinea-Gonzalez \cite{CG90} and Chinea-Marrero \cite{CM92}. For example, any almost contact metric $5$-manifold of Chinea-Marrero class $\mca{N}_2$ is of class $\W_1\op\W_3\op\W_5\op\W_6\op\W_8\op\W_9$ and vice versa (see theorem \ref{thm:1}).

Thirdly, we determine necessary and sufficient conditions for the existence of a compatible connection $\nabla^c$ with vectorial, totally skew-symmetric or traceless cyclic torsion (see \autoref{sec:4}). If the torsion tensor of $\nabla^c$ is traceless cyclic, then (see proposition \ref{prop:3}) the almost contact metric $5$-manifold is of class
\be
\W_2\op\W_3\op\W_5\op\W_6\op\W_7\op\W_8\op\W_9\op\W_{10}.
\ee
Conversely, any almost contact metric $5$-manifold of this class admits a unique metric connection $\nabla^c$ with traceless cyclic torsion that is compatible with the underlying structure (see\ theorem \ref{thm:5}). Theorems \ref{thm:6} and \ref{thm:4} contain similar results for the cases of vectorial and totally skew-symmetric torsion, the respective types of the intrinsic torsion tensor are $\W_1\op\W_2$ and $\W_3\op\W_4\op\W_5\op\W_6$.

Finally, we present explicit examples of almost contact metric $5$-manifolds (see \autoref{sec:5}). The corresponding intrinsic torsion tensors are of type
\be
\W_1,\W_2,\W_3,\W_5,\W_6,\W_8,\W_9,\W_{10}.
\ee
Using the results of \autoref{sec:4}, we identify compatible connections $\nabla^c$ for each example. The torsion tensors of these connections realize the following types:
\ba
&\T_{1,1},\T_{1,2}, &
&\T_{2,1},\T_{2,3},\T_{2,4}, &
&\T_{3,1},\T_{3,3},\ldots,\T_{3,8}.
\ea

Almost contact metric $5$-manifolds of class $\W_4\op\W_7$ exist, too. Indeed, Blair constructed an almost contact metric structure on $\mrm{S}^5$ which is nearly cosymplectic (see \cite{Bla71}). A glance at theorem \ref{thm:2} allows to deduce that this almost contact metric $5$-manifold is of class $\W_4\op\W_7$.
%
%
%
%
\section{The local model}\label{sec:2}\noindent
We first introduce some notation. $\R^5$ denotes the $5$-dimensional Euclidean space. We fix an orientation in $\R^5$ and use its scalar product $\lan \cdot,\cdot\ran$ to identify $\R^5$ with its dual space ${\R^5}^{\ast}$. Let $\lb(e_1,\ldots,e_5\rb)$ denote an oriented orthonormal basis and $\Lambda^k$ the space of $k$-forms of $\R^5$. The family of operators
\be
\sigma_j:\Lambda^k\x\Lambda^l\ra\Lambda^{k+l-2j},
\ee
\ba
\sigma_{j}\lb(\alpha,\beta\rb)&:=\sum_{i_1<\ldots<i_j}\lb(e_{i_1}\hook\ldots\hook e_{i_j}\hook\alpha\rb)\wedge\lb(e_{i_1}\hook\ldots\hook e_{i_j}\hook\beta\rb), &
\sigma_0\lb(\alpha,\beta\rb)&:=\alpha\wedge\beta
\ea
allows us to define an inner product and a norm on $\Lambda^k$ as
\ba
\lan\alpha,\beta\ran&:=\sigma_k\lb(\alpha,\beta\rb), &
\|\alpha\|&:=\sqrt{\sigma_k\lb(\alpha,\alpha\rb)}.
\ea
The special orthogonal group $\SO(5)$ acts on $\Lambda^k$ via the adjoint representation $\vrho$. The differential
\be
\vrho_* : \mfr{so}(5) \ra \mfr{so}\lb(\Lambda^k\rb)
\ee
of this faithful representation can be expressed as
\be
\vrho_\ast\lb(\omega\rb)\lb(\alpha\rb) = \sigma_1\lb(\omega,\alpha\rb)
\ee
by identifying the Lie algebra $\so(5)$ with the space of $2$-forms $\Lambda^2$.

We now define the vector $\xi:=e_5$, denote its dual $1$-form by $\eta$ and consider the endomorphism $\vphi:\R^5\ra\R^5$ given by the matrix
\be
\begin{pmatrix}
0 & 1 & 0 & 0 & 0\\
-1 & 0 & 0 & 0 & 0\\
0 & 0 & 0 & 1 & 0\\
0 & 0 & -1 & 0 & 0\\
0 & 0 & 0 & 0 & 0
\end{pmatrix}.
\ee
The latter satisfies
\ba
\vphi\lb(\xi\rb)&=0, & \vphi^2&=-\Id+\eta\ox\xi, & \lan\vphi\lb(X\rb),\vphi\lb(Y\rb)\ran=\lan X,Y\ran-\eta\lb(X\rb)\eta\lb(Y\rb).
\ea
The $4$-dimensional, compact, connected Lie group $\Un(2)\subset\SO(5)$ can be described as the isotropy group of the $2$-form $\Phi$ defined via
\be
\Phi\lb(X,Y\rb):=\lan X,\vphi\lb(Y\rb)\ran.
\ee
Alternatively, we have
\be
\Un(2)=\lb\{A\in\Orth(5)\setsep A\lb(\xi\rb)=\xi,\ A\circ\vphi=\vphi\circ A\rb\}.
\ee
%
%
\subsection{The decomposition of \texorpdfstring{$\Lambda^k$}{Lambda^k}}
The Hodge operator $\ast:\Lambda^k\ra\Lambda^{5-k}$ is $\Un(2)$-equi\-var\-i\-ant. Using this operator, we decompose $\Lambda^k$ into irreducible $\Un(2)$-modules of real type. The space
\be
\Lambda^1=\Lambda^1_1\op\Lambda^1_2
\ee
splits into the two irreducible $\Un(2)$-modules
\ba
\Lambda^1_1 &:= \lb\{t\cdot\eta\setsep t\in\R\rb\}, &
\Lambda^1_2 &:= \lb\{\alpha\in\Lambda^1\setsep\xi\hook\alpha=0\rb\}.
\ea
The space of $2$-forms
\be
\Lambda^2=\Lambda^2_1\op\Lambda^2_2\op\Lambda^2_3\op\Lambda^2_4
\ee
decomposes into four irreducible $\Un(2)$-modules:
\ba
\Lambda^2_1 &:= \lb\{t\cdot\Phi\setsep t\in\R\rb\}, \\
\Lambda^2_2 &:= \lb\{\alpha\in\Lambda^2\setsep \Phi\wedge\alpha=0,\  \ast\alpha=\eta\wedge\alpha\rb\},\\
\Lambda^2_3 &:= \lb\{\alpha\in\Lambda^2\setsep \ast\alpha=-\eta\wedge\alpha\rb\},\\
\Lambda^2_4 &:= \lb\{\alpha\in\Lambda^2\setsep \eta\wedge\alpha=0\rb\}.
\ea
The dimensions of these modules are
\ba
\dim\lb(\Lambda^2_i\rb)&=i.
\ea
Moreover, we have
\ba
\Lambda^2_1\op\Lambda^2_2 & =
\lb\{\alpha\in\Lambda^2\setsep\ast\alpha=\eta\wedge\alpha\rb\}, &
\Lambda^2_2\op\Lambda^2_3 & =\lb\{\alpha\in\Lambda^2\setsep\Phi\wedge\alpha=0\rb\}
\ea
and
\be
\alpha\lb(\vphi\lb(X\rb),\vphi\lb(Y\rb)\rb)=\begin{cases}
  \alpha\lb(X,Y\rb) & \text{iff $\alpha\in\Lambda^2_1\op\Lambda^2_3$}\\
  -\alpha\lb(X,Y\rb) & \text{iff $\alpha\in\Lambda^2_2$}\\
  0 & \text{iff $\alpha\in\Lambda^2_4$}\ .
\end{cases}
\ee
We then define
\ba
\Lambda^3_i &:= \ast\Lambda^2_i, & \Lambda^4_i &:= \ast\Lambda^1_i.
\ea
Consequently, the decompositions
\ba
\Lambda^3&=\Lambda^3_1\op\Lambda^3_2\op\Lambda^3_3\op\Lambda^3_4, &
\Lambda^4=\Lambda^4_1\op\Lambda^4_2
\ea
split the spaces of $3$-forms and $4$-forms of $\R^5$ into irreducible $\Un(2)$-modules.

The Lie algebra $\so\lb(5\rb)$ splits into $\un\lb(2\rb)=\Lambda^2_1\op\Lambda^2_3$, spanned by
\ba
\Phi&=e_1\wedge e_2+e_3\wedge e_4,&
\omega_1&:=e_1\wedge e_2-e_3\wedge e_4,\\
\omega_2&:=e_1\wedge e_3+e_2\wedge e_4,&
\omega_3&:=e_1\wedge e_4-e_2\wedge e_3,
\ea
and its orthogonal complement $\m:=\Lambda^2_2\op\Lambda^2_4$.
%
%
\subsection{The decomposition of \texorpdfstring{$\Lambda^1\ox\Lambda^2$}{Lambda^1xLambda^2}, \texorpdfstring{$\Lambda^2\ox\Lambda^1$}{Lambda^2xLambda^1} and \texorpdfstring{$\Lambda^1\ox\m$}{Lambda^1xm}}\label{sec:2.2}
The space
\be
\T:=\Lambda^2\ox\Lambda^1=\lb\{T\in\ox^3\Lambda^1\setsep T\lb(X,Y,Z\rb)+ T\lb(Y,X,Z\rb)=0\rb\}
\ee
splits into three irreducible $\Orth(5)$-modules (see \cite{Car25}),
\be
\T=\T_1\op\T_2\op\T_3,
\ee
where
\ba
\T_1&:=\lb\{ T\in\T\setsep\ex\alpha\in\Lambda^1: T\lb(X,Y,Z\rb)
=\alpha\lb(X\rb)\lan Y,Z\ran-\alpha\lb(Y\rb)\lan X,Z\ran\rb\},\\
\T_2&:=\lb\{ T\in\T\setsep T\lb(X,Y,Z\rb)+ T\lb(X,Z,Y\rb)=0\rb\},\\
\T_3&:=\lb\{ T\in\T\setsep\mfr{S}_{X,Y,Z} T\lb(X,Y,Z\rb)=0,\  \textstyle{\sum_i} T\lb(X,e_i,e_i\rb)=0\rb\}.
\ea
Here $\mfr{S}_{X,Y,Z}$ denotes the cyclic sum over $X,Y,Z$. There exists an $\Orth(5)$-equivariant bijection $\tau$ between
\be
\A:=\Lambda^1\ox\Lambda^2=\lb\{ A\in\ox^3\Lambda^1\setsep A\lb(X,Y,Z\rb)+ A\lb(X,Z,Y\rb)=0\rb\}
\ee
and $\T$ (see \cite{Car25}) explicitly given by
\ba
\tau\lb( A\rb)\left(X,Y,Z\right)&= A\left(X,Y,Z\right)- A\left(Y,X,Z\right),\\
2\,\tau^{-1}\lb( T\rb)\left(X,Y,Z\right)&= T\left(X,Y,Z\right)- T\left(Y,Z,X\right)+ T\left(Z,X,Y\right).
\ea
Using $\tau$, we obtain that
\be
\A=\A_1\op\A_2\op\A_3
\ee
splits into the three irreducible $\Orth(5)$-modules
\ba
\A_1&:=\lb\{ A\in\A\setsep\ex\alpha\in\Lambda^1: A\lb(X,Y,Z\rb)
=\alpha\lb(Z\rb)\lan X,Y\ran-\alpha\lb(Y\rb)\lan X,Z\ran\rb\},\\
\A_2&:=\lb\{ A\in\A\setsep A\lb(X,Y,Z\rb)+ A\lb(Y,X,Z\rb)=0\rb\},\\
\A_3&:=\lb\{ A\in\A\setsep\mfr{S}_{X,Y,Z} A\lb(X,Y,Z\rb)=0,\  \textstyle{\sum_i} A\lb(e_i,e_i,X\rb)=0\rb\}.
\ea

We now decompose these three spaces under the action of $\Un(2)$. For this purpose, we define the injective $\Un(2)$-equivariant maps
\ba
\theta_1&:\Lambda^1\ra\A_1, &
\theta_1&\lb(\alpha\rb)\lb(X,Y,Z\rb):=\alpha\lb(Z\rb)\lan X,Y\ran-\alpha\lb(Y\rb)\lan X,Z\ran,\\
\theta_2&:\Lambda^3\ra\A_2, &
\theta_2&\lb(\alpha\rb):=\sum_ie_i\ox\lb(e_i\hook\alpha\rb),\\
\theta_3&:\Lambda^2_1\op\Lambda^2_2\op\Lambda^2_3\ra\A_3, &
\theta_3&\lb(\alpha\rb):=3\,\eta\ox\alpha-\theta_2(\eta\wedge\alpha)
\ea
and
\ba
\theta_4&:\Lambda^1_2\ra\A_3, &
\theta_4&\lb(\alpha\rb):=\sum_ie_i\ox\lb(\alpha\wedge e_i\rb)+\frac{1}{2}\theta_2(\alpha\hook\lb(\Phi\wedge\Phi\rb))-3\lb(\alpha\hook\Phi\rb)\ox\Phi-\eta\ox\lb(\alpha\wedge\eta\rb),\\
\theta_5&:\Lambda^1_2\ra\A_3, &
\theta_5&\lb(\alpha\rb):=\sum_ie_i\ox\lb(\alpha\wedge e_i\rb)+\theta_2(\alpha\hook\lb(\Phi\wedge\Phi\rb))-6\lb(\alpha\hook\Phi\rb)\ox\Phi+2\,\eta\ox\lb(\alpha\wedge\eta\rb),
\ea
and introduce the following subspaces of $\A$:
\ba
\A_{1,i}&:=\theta_1\lb(\Lambda^1_i\rb),\quad i=1,2, \\
\A_{2,i}&:=\theta_2\lb(\Lambda^3_i\rb),\quad i=1,2,3,4, \\
\A_{3,i}&:=\theta_3\lb(\Lambda^2_i\rb),\quad i=1,2,3, \\
\A_{3,i}&:=\theta_i\lb(\Lambda^1_2\rb),\quad i=4,5, \\
\A_{3,6}&:=\lb\{A\in\A_3\setsep A\lb(X,Y,Z\rb)=
A\lb(\vphi\lb(X\rb),Y,\vphi\lb(Z\rb)\rb)+A\lb(\vphi\lb(X\rb),\vphi\lb(Y\rb),Z\rb)\rb\}, \\
\A_{3,7}&:=\lb\{A\in\A_3\setsep A\lb(X,Y,Z\rb)=-
A\lb(X,\vphi\lb(Y\rb),\vphi\lb(Z\rb)\rb)\rb\}, \\
\A_{3,8}&:=\lb\{A\in\A_3\setsep A\lb(X,Y,Z\rb)=-
A\lb(\vphi\lb(X\rb),Y,\vphi\lb(Z\rb)\rb)-A\lb(\vphi\lb(X\rb),\vphi\lb(Y\rb),Z\rb)\rb\}, \\
\A_{3,9}&:=\lb\{A\in\A_3\setsep A\lb(X,Y,Z\rb)=
A\lb(X,\vphi\lb(Y\rb),\vphi\lb(Z\rb)\rb)\rb\}.
\ea
Inspecting the latter, we deduce
\begin{prop}
The space $\A=\Lambda^1\ox\Lambda^2$ splits into $15$ irreducible $\Un(2)$-modules,
\be
\A=\A_{1,1}\op\A_{1,2}\op\A_{2,1}\op\ldots\op\A_{2,4}\op\A_{3,1}\op\ldots\op\A_{3,9},
\ee
\ba
\A_1&=\A_{1,1}\op\A_{1,2}, &
\A_2&=\A_{2,1}\op\ldots\op\A_{2,4}, &
\A_3&=\A_{3,1}\op\ldots\op\A_{3,9}.
\ea
Moreover,
\bite
\item[a)] The dimensions of the $\Un(2)$-modules are
\ba
\dim\lb(\A_{1,1}\rb)&=\dim\lb(\A_{2,1}\rb)=\dim\lb(\A_{3,1}\rb)=1,\\
\dim\lb(\A_{2,2}\rb)&=\dim\lb(\A_{3,2}\rb)=2,\\
\dim\lb(\A_{2,3}\rb)&=\dim\lb(\A_{3,3}\rb)=\dim\lb(\A_{3,6}\rb)=3,\\
\dim\lb(\A_{1,2}\rb)&=\dim\lb(\A_{2,4}\rb)=\dim\lb(\A_{3,4}\rb)=\dim\lb(\A_{3,5}\rb)=\dim\lb(\A_{3,7}\rb)=4,\\
\dim\lb(\A_{3,8}\rb)&=6,\\
\dim\lb(\A_{3,9}\rb)&=8.\\
\ea
\item[b)] The following $\Un(2)$-modules are isomorphic:
\ba
\A_{1,1}&\cong\A_{2,1}\cong\A_{3,1}, & \A_{2,2}&\cong\A_{3,2}, \\
\A_{2,3}&\cong\A_{3,3}\cong\A_{3,6}, & \A_{1,2}&\cong\A_{2,4}\cong\A_{3,4}\cong\A_{3,5}.
\ea
\item[c)] The $\Un(2)$-modules $\A_{3,7}$ and $\A_{1,2}\cong\A_{2,4}\cong\A_{3,4}\cong\A_{3,5}$ are not isomorphic.
\eite
\end{prop}\noindent
We recommend the article \cite{Fin95} for a qualitative decomposition of $\A$ in the style of the book \cite{Sal89}. Comparing dimensions and multiplicities, we have the following isomorphisms between the spaces of \cite{Fin95} and the modules $\A_{i,j}$ above:
\ba
\A_{1,1}&\cong\A_{2,1}\cong\A_{3,1}\cong\R, &
\A_{1,2}&\cong\A_{2,4}\cong\A_{3,4}\cong\A_{3,5}\cong
\lb\llbracket\lambda^{1,0}\rb\rrbracket, \\
\A_{2,2}&\cong\A_{3,2}\cong\lb\llbracket\lambda^{2,0}\rb\rrbracket, & 
\A_{2,3}&\cong\A_{3,3}\cong\A_{3,6}\cong\lb[\lambda^{1,1}_0\rb], \\
\A_{3,7}&\cong\lb\llbracket A\rb\rrbracket, &
\A_{3,8}&\cong\lb\llbracket\sigma^{2,0}\rb\rrbracket, \\
\A_{3,9}&\cong\lb\llbracket B\rb\rrbracket.
\ea

The bijection $\tau:\A\ra\T$ is $\Un(2)$-equivariant. Defining the subspaces
\ba
\T_{1,i}&:=\tau\lb(\A_{1,i}\rb), &
\T_{2,i}&:=\tau\lb(\A_{2,i}\rb), &
\T_{3,i}&:=\tau\lb(\A_{3,i}\rb)
\ea
of $\T$, we consequently have
\begin{cor}\label{cor:1}
The space $\T=\Lambda^2\ox\Lambda^1$ splits into $15$ irreducible $\Un(2)$-modules,
\be
\T=\T_{1,1}\op\T_{1,2}\op\T_{2,1}\op\ldots\op\T_{2,4}\op\T_{3,1}\op\ldots\op\T_{3,9},
\ee
\ba
\T_1&=\T_{1,1}\op\T_{1,2}, &
\T_2&=\T_{2,1}\op\ldots\op\T_{2,4}, &
\T_3&=\T_{3,1}\op\ldots\op\T_{3,9}.
\ea
\end{cor}
Finally, we consider the space $\W:=\Lambda^1\ox\m$. Using the projection $\pr_{\m}$ onto $\m$, we define the map $\pr_{\W}:\A\ra\W$ via
\be
\pr_{\W}\lb(\alpha\ox\beta\rb):=\alpha\ox\pr_{\m}\lb(\beta\rb).
\ee
\begin{lem}\label{lem:1}
The map $\pr_{\W}$ is $\Un(2)$-equivariant and satisfies
\ba
\pr_{\W}\lb(\A_{1,1}\rb)&=\A_{1,1}, & 
\pr_{\W}\lb(\A_{1,2}\rb)&=\pr_{\W}\lb(\A_{3,5}\rb), &
\pr_{\W}\lb(\A_{2,1}\rb)&=\pr_{\W}\lb(\A_{3,1}\rb), \\
\pr_{\W}\lb(\A_{2,2}\rb)&=\A_{2,2}, &
\pr_{\W}\lb(\A_{2,3}\rb)&=\pr_{\W}\lb(\A_{3,3}\rb), &
\pr_{\W}\lb(\A_{2,4}\rb)&=\pr_{\W}\lb(\A_{3,4}\rb), \\
\pr_{\W}\lb(\A_{3,2}\rb)&=\A_{3,2}, &
\pr_{\W}\lb(\A_{3,6}\rb)&=\A_{3,6}, &
\pr_{\W}\lb(\A_{3,7}\rb)&=\A_{3,7}, \\
\pr_{\W}\lb(\A_{3,8}\rb)&=\A_{3,8}, &
\pr_{\W}\lb(\A_{3,9}\rb)&=\lb\{0\rb\}. &&
\ea
\end{lem}\noindent
This lemma, together with the definition of
\ba
\W_1&:=\A_{1,1}, &
\W_2&:=\pr_{\W}\lb(\A_{1,2}\rb), &
\W_3&:=\pr_{\W}\lb(\A_{2,1}\rb), &
\W_4&:=\A_{2,2}, \\
\W_5&:=\pr_{\W}\lb(\A_{2,3}\rb), &
\W_6&:=\pr_{\W}\lb(\A_{2,4}\rb), &
\W_7&:=\A_{3,2}, &
\W_8&:=\A_{3,6}, \\
\W_9&:=\A_{3,7}, &
\W_{10}&:=\A_{3,8},
\ea
leads to
\begin{prop}\label{prop:1}
The space $\W=\Lambda^1\ox\m$ splits into $10$ irreducible $\Un(2)$-modules:
\be
\W=\W_1\op\ldots\op\W_{10}.
\ee
\end{prop}
%
%
%
%
\section{Almost contact metric structures}\label{sec:3}\noindent
Let $\lb(M^{2k+1},g\rb)$ be a $(2k+1)$-dimensional Riemannian manifold. An \emph{almost contact metric structure} on $\lb(M^{2k+1},g\rb)$ consists of a vector field $\xi$ of length one, its dual $1$-form $\eta$ and an endomorphism $\vphi$ of the tangent bundle such that
\ba
\vphi\lb(\xi\rb)&=0, & \vphi^2&=-\Id+\eta\ox\xi, &
g\lb(\vphi\lb(X\rb),\vphi\lb(Y\rb)\rb)=g\lb( X,Y\rb)-\eta\lb(X\rb)\eta\lb(Y\rb).
\ea
Equivalently, these structures can be defined as a reduction of the structure group of orthonormal frames of the tangent bundle to $\Un\left(k\right)$ (see \cites{Gra59,SH61,SH62}). The \emph{fundamental form} $\Phi$ of an \emph{almost contact metric manifold} $\lb(M^{2k+1},g,\xi,\eta,\vphi\rb)$ is a $2$-form defined by
\be
\Phi\lb(X,Y\rb):=g\lb(X,\vphi\lb(Y\rb)\rb). 
\ee

Consider an almost contact metric $5$-manifold $\lb(M^5,g,\xi,\eta,\vphi\rb)$. The corresponding fundamental form satisfies $\eta\wedge\Phi\wedge\Phi\neq0$. Consequently, there exists an oriented orthonormal frame $\left(e_1,\ldots,e_5\right)$ realizing the local model introduced in \autoref{sec:2}, i.e.\
\ba
\xi&=e_5, & \Phi&=e_{12}+e_{34}.
\ea
Here and henceforth we identify $TM^5$ with its dual space ${TM^5}^{\ast}$ using $g$. Moreover, we use the notation $e_{i_1\ldots i_j}$ for the exterior product
\be
e_{i_1}\wedge\ldots\wedge e_{i_j}.
\ee
We call $\left(e_1,\ldots,e_5\right)$ an \emph{adapted frame} of the almost contact metric manifold. The connection forms
\be
\omega_{ij}^g:=g\lb(\nabla^{g}e_i,e_j\rb)
\ee
of the Levi-Civita connection $\nabla^g$ define a $1$-form 
\be
\Omega^g:=\lb(\omega_{ij}^g\rb)_{1\leq i,j\leq 5}
\ee
with values in the Lie algebra $\so(5)$. We define the \emph{intrinsic torsion} $\Gamma$ of the almost contact metric manifold as
\be
\Gamma := \pr_\m\lb(\Omega^g\rb).
\ee
Since the Riemannian covariant derivative of the fundamental form $\Phi$ is given by
\be
\nabla^{g}\Phi = \vrho_\ast\lb(\Gamma\rb)\lb(\Phi\rb),
\ee
almost contact metric manifolds can be classified according to the algebraic type of $\Gamma$ with respect to the decomposition of $\Lambda^1\ox\m$ into irreducible $\Un(2)$-modules (cf.\ \cite{Fri03}). Applying proposition \ref{prop:1}, we split $\Gamma$ as
\be
\Gamma = \Gamma_1+\ldots+\Gamma_{10},
\ee
to the effect that $2^{10}=1024$ classes arise. We say that the almost contact metric manifold \emph{is of class} $\W_{i_1}\op\ldots\op\W_{i_k}$ if
\be
\Gamma\in\W_{i_1}\op\ldots\op\W_{i_k}.
\ee
Moreover, $\lb(M^5,g,\xi,\eta,\vphi\rb)$ \emph{is of strict class} $\W_{i_1}\op\ldots\op\W_{i_k}$ if it is of class $\W_{i_1}\op\ldots\op\W_{i_k}$ and $\Gamma_{i_j}\neq0$. Almost contact metric manifolds with $\Gamma=0$ are called \emph{integrable}. The \emph{Nijenhuis tensor} $N$ of $\lb(M^5,g,\xi,\eta,\vphi\rb)$ is defined by
\be
N\lb(X,Y\rb):=\lb[\vphi,\vphi\rb]\lb(X,Y\rb)+d\eta\lb(X,Y\rb)\cdot\xi, 
\ee
where $\lb[\vphi,\vphi\rb]$ is the Nijenhuis torsion of $\vphi$,
\be
\lb[\vphi,\vphi\rb]\lb(X,Y\rb)=
\lb[\vphi\lb(X\rb),\vphi\lb(Y\rb)\rb]
+\vphi^2\lb(\lb[X,Y\rb]\rb)
-\vphi\lb(\lb[\vphi\lb(X\rb),Y\rb]\rb)
-\vphi\lb(\lb[X,\vphi\lb(Y\rb)\rb]\rb),
\ee
and the differential $d\alpha$ and the co-differential $\delta\alpha$ of a $k$-form $\alpha$ are given by
\ba
d\alpha&=\sum_i e_i\wedge\nabla^{g}_{e_i}\alpha, &
\delta\alpha&=-\sum_i e_i\hook\nabla^{g}_{e_i}\alpha.
\ea
%
%
\subsection{Classification schemes}\label{sec:3.1}
We classified $\lb(M^5,g,\xi,\eta,\vphi\rb)$ with respect to the algebraic type of its intrinsic torsion tensor $\Gamma$. In this subsection, we relate this scheme to the Chinea-Gonzalez classification \cite{CG90} and the Chinea-Marrero classification \cite{CM92} of almost contact metric manifolds. Motivated by
\be
\lb(\nabla^g_X\Phi\rb)\lb(Y,Z\rb)=-
\lb(\nabla^g_X\Phi\rb)\lb(\vphi\lb(Y\rb),\vphi\lb(Z\rb)\rb)+\eta\lb(Y\rb)\lb(\nabla^g_X\Phi\rb)\lb(\xi,Z\rb)\\
+\eta\lb(Z\rb)\lb(\nabla^g_X\Phi\rb)\lb(Y,\xi\rb),
\ee
the authors of \cite{CG90} decompose the subspace
\ba
\mca{C}=\lb\{A\in\Lambda^1\ox\Lambda^2\setsep A\lb(X,Y,Z\rb)=-
A\lb(X,\vphi\lb(Y\rb),\vphi\lb(Z\rb)\rb)\rb.&+\eta\lb(Y\rb)A\lb(X,\xi,Z\rb)\\
&\lb.+\eta\lb(Z\rb)A\lb(X,Y,\xi\rb)\rb\}
\ea
of $\Lambda^1\ox\Lambda^2$ into $10$ irreducible $\Un(2)$-modules:
\be
\mca{C}=\mca{C}_2\op\mca{C}_4\op\ldots\op\mca{C}_{12}.
\ee
Although $\mca{C}$ coincides with $\W$, this decomposition differs from the one in proposition \ref{prop:1}. We have the following isomorphisms between the spaces $\mca{C}_2,\mca{C}_4,\ldots,\mca{C}_{12}$ of \cite{CG90} and $\W_1,\ldots,\W_{10}$:
\ba
\W_1&\cong\W_3\cong\mca{C}_5\cong\mca{C}_6, &
\W_2&\cong\W_6\cong\mca{C}_4\cong\mca{C}_{12}, &
\W_4&\cong\W_7\cong\mca{C}_{10}\cong\mca{C}_{11}, \\
\W_5&\cong\W_8\cong\mca{C}_7\cong\mca{C}_8, &
\W_9&\cong\mca{C}_2, &
\W_{10}&\cong\mca{C}_9.
\ea
Moreover, $\lb(M^5,g,\xi,\eta,\vphi\rb)$ is said to be of class $\mca{C}_{i_1}\op\ldots\op\mca{C}_{i_k}$ if
\be
\nabla^{g}\Phi\in\mca{C}_{i_1}\op\ldots\op\mca{C}_{i_k}.
\ee
\begin{thm}
Let $\lb(M^5,g,\xi,\eta,\vphi\rb)$ be a $5$-dimensional almost contact metric manifold.
Then the following equivalences hold:
\btab{c|c|c|c|c|c|c|c|c|c}
$\lb(M^5,g,\xi,\eta,\vphi\rb)$ is of class & $\mca{C}_2$ & $\mca{C}_4$ & $\mca{C}_5$
& $\mca{C}_6$ & $\mca{C}_7$ & $\mca{C}_8$ & $\mca{C}_9$ & $\mca{C}_{10}\op\mca{C}_{11}$ & $\mca{C}_4\op\mca{C}_{12}$\\\hline
if and only if it is of class & $\W_9$ & $\W_6$ & $\W_1$ & $\W_3$ & $\W_5$
& $\W_8$ & $\W_{10}$ & $\W_4\op\W_7$ & $\W_2\op\W_6$ 
\etab
Moreover, if $\lb(M^5,g,\xi,\eta,\vphi\rb)$ is not integrable and either of class $\mca{C}_{10}$ or of class $\mca{C}_{11}$, then $\lb(M^5,g,\xi,\eta,\vphi\rb)$ is of strict class $\W_4\op\W_7$. Any non-integrable almost contact metric $5$-manifold of class $\mca{C}_{12}$ is of strict class $\W_2\op\W_6$.
\end{thm}
\begin{proof}
The results are a direct consequence of
\be\label{eqn:1}\tag{$\star$}
\nabla^{g}_X\Phi = \vrho_\ast\lb(\Gamma\lb(X\rb)\rb)\lb(\Phi\rb)
=\sigma_1\lb(\Gamma\lb(X\rb),\Phi\rb),
\ee
valid for any $5$-dimensional almost contact metric manifold.
\end{proof}
Viewed as a $\left(3,0\right)$-tensor,
\be
N\lb(X,Y,Z\rb):=g\lb(X,N\lb(Y,Z\rb)\rb), 
\ee
the Nijenhuis tensor satisfies
\ba
N\lb(X,Y,Z\rb)&=-N\lb(X,\vphi\lb(Y\rb),\vphi\lb(Z\rb)\rb)+\eta\lb(Y\rb)N\lb(X,\xi,Z\rb)+\eta\lb(Z\rb)N\lb(X,Y,\xi\rb),\\
N\lb(X,Y,Z\rb)&=-N\lb(\vphi\lb(X\rb),\vphi\lb(Y\rb),Z\rb)+\eta\lb(X\rb)N\lb(\xi,Y,Z\rb)-\eta\lb(Y\rb)N\lb(\vphi\lb(X\rb),\xi,\vphi\lb(Z\rb)\rb).
\ea
The subspace $\mca{N}$ of tensors of $\Lambda^1\ox\Lambda^2$ satisfying these two conditions splits into five irreducible $\Un(2)$-modules (see \cite{CM92}),
\be
\mca{N}=\mca{N}_2\op\ldots\op\mca{N}_6.
\ee
Moreover, the authors of \cite{CM92} define $\lb(M^5,g,\xi,\eta,\vphi\rb)$ to be of class $\mca{N}_{i_1}\op\ldots\op\mca{N}_{i_k}$ if
\be
N\in\mca{N}_{i_1}\op\ldots\op\mca{N}_{i_k}.
\ee
Before we describe these classes in terms of the intrinsic torsion, we prove
\begin{lem}\label{lem:2}
The following formulae hold on $5$-dimensional almost contact metric manifolds $\lb(M^5,g,\xi,\eta,\vphi\rb)$:
\be
g\lb(\lb(\nabla^g_X\vphi\rb)\lb(Y\rb),Z\rb)=\lb(\nabla^g_X\Phi\rb)\lb(Z,Y\rb),
\ee
\be
\lb(\nabla^g_X\eta\rb)\lb(Y\rb)=g\lb(\nabla^g_X\xi,Y\rb)=\lb(\nabla^g_X\Phi\rb)\lb(\xi,\vphi\lb(Y\rb)\rb).
\ee
\end{lem}
\begin{proof}
We compute
\ba
g\lb(\nabla^g_X\xi,Y\rb)&=g\lb(\nabla^g_X\xi,Y\rb)-\eta\lb(Y\rb)g\lb(\nabla^g_X\xi,\xi\rb)\\
&=g\lb(\nabla^g_X\xi,Y\rb)-\eta\lb(Y\rb)\eta\lb(\nabla^g_X\xi\rb)\\
&=g\lb(\vphi\lb(\nabla^g_X\xi\rb),\vphi\lb(Y\rb)\rb)\\
&=g\lb(\nabla^g_X\lb(\vphi\lb(\xi\rb)\rb),\vphi\lb(Y\rb)\rb)-g\lb(\lb(\nabla^g_X\vphi\rb)\lb(\xi\rb),\vphi\lb(Y\rb)\rb)\\
&=-g\lb(\lb(\nabla^g_X\vphi\rb)\lb(\xi\rb),\vphi\lb(Y\rb)\rb).
\ea
Moreover, we have
\ba
g\lb(\lb(\nabla^g_X\vphi\rb)\lb(Y\rb),Z\rb)&=g\lb(\nabla^g_X\lb(\vphi\lb(Y\rb)\rb),Z\rb)
-g\lb(\vphi\lb(\nabla^g_XY\rb),Z\rb)\\
&=X\lb(g\lb(\vphi\lb(Y\rb),Z\rb)\rb)-g\lb(\vphi\lb(Y\rb),\nabla^g_XZ\rb)-g\lb(\vphi\lb(\nabla^g_XY\rb),Z\rb)\\
&=X\lb(\Phi\lb(Z,Y\rb)\rb)-\Phi\lb(\nabla^g_XZ,Y\rb)-\Phi\lb(Z,\nabla^g_XY\rb)\\
&=\lb(\nabla^g_X\Phi\rb)\lb(Z,Y\rb).
\ea
Combining these two equations, we obtain the desired result.
\end{proof}\noindent
\begin{thm}\label{thm:1}
Let $\lb(M^5,g,\xi,\eta,\vphi\rb)$ be a $5$-dimensional almost contact metric manifold.
Then the following equivalences hold:
\btab{c|c}
$\lb(M^5,g,\xi,\eta,\vphi\rb)$ is of class & if and only if it is of class \\\hline
$\mca{N}_2$ & $\W_1\op\W_3\op\W_5\op\W_6\op\W_8\op\W_9$ \\\hline
$\mca{N}_3$ & $\W_1\op\W_3\op\W_5\op\W_6\op\W_8\op\W_{10}$ \\\hline
$\mca{N}_4\op\mca{N}_5$ & $\W_1\op\W_3\op\W_4\op\W_5\op\W_6\op\W_7\op\W_8$ \\\hline
$\mca{N}_6$ & $\W_1\op\W_2\op\W_3\op\W_5\op\W_6\op\W_8$
\etab
In particular, the Nijenhuis tensor of $\lb(M^5,g,\xi,\eta,\vphi\rb)$ vanishes if and only if the almost contact metric manifold is of class $\W_1\op\W_3\op\W_5\op\W_6\op\W_8$.
\end{thm}
\begin{proof}
Utilizing lemma \ref{lem:2}, we compute
\ba
N\lb(X,Y,Z\rb)=&g\lb(X,\lb[\vphi\lb(Y\rb),\vphi\lb(Z\rb)\rb]
+\vphi^2\lb(\lb[Y,Z\rb]\rb)
-\vphi\lb(\lb[\vphi\lb(Y\rb),Z\rb]\rb)
-\vphi\lb(\lb[Y,\vphi\lb(Z\rb)\rb]\rb)\rb)\\
&+g\lb(X,d\eta\lb(Y,Z\rb)\cdot\xi\rb)\\
=&g\lb(X,\lb(\nabla^g_{\vphi\lb(Y\rb)}\vphi\rb)\lb(Z\rb)
-\lb(\nabla^g_{\vphi\lb(Z\rb)}\vphi\rb)\lb(Y\rb)
+\vphi\lb(\lb(\nabla^g_Z\vphi\rb)\lb(Y\rb)-\lb(\nabla^g_Y\vphi\rb)\lb(Z\rb)\rb)\rb)\\
&+g\lb(X,\xi\rb)\lb(\lb(\nabla^g_Y\eta\rb)\lb(Z\rb)-\lb(\nabla^g_Z\eta\rb)\lb(Y\rb)\rb)\\
=&\lb(\nabla^g_{\vphi\lb(Y\rb)}\Phi\rb)\lb(X,Z\rb)
-\lb(\nabla^g_{\vphi\lb(Z\rb)}\Phi\rb)\lb(X,Y\rb)
+\lb(\nabla^g_Y\Phi\rb)\lb(\vphi\lb(X\rb),Z\rb)\\
&-\lb(\nabla^g_Z\Phi\rb)\lb(\vphi\lb(X\rb),Y\rb)+\eta\lb(X\rb)\lb(\nabla^g_{Y}\Phi\rb)\lb(\xi,\vphi\lb(Z\rb)\rb)
-\eta\lb(X\rb)\lb(\nabla^g_{Z}\Phi\rb)\lb(\xi,\vphi\lb(Y\rb)\rb).
\ea
This equation, together with formula (\ref{eqn:1}), enables us to express the Nijenhuis tensor in terms of $\Gamma$.
\end{proof}
%
%
\subsection{Special types}\label{sec:3.2}
As indicated above, there are many types of almost contact metric manifolds. We introduce those carrying names and review them in the light of our classification scheme.

An almost contact metric manifold $\lb(M^5,g,\xi,\eta,\vphi\rb)$ is said to be
\bite
\item\emph{normal} (see \cites{SH61,SH62}) if its Nijenhuis tensor vanishes.
\item\emph{almost co-Kähler} (see \cite{JV81}) or \emph{almost cosymplectic} (see \cite{Oub85}) if
\ba
d\Phi&=0, & d\eta&=0.
\ea
In overlap to our next notion, the authors of \cites{JV81,Lib59} also use the term cosymplectic in this case.
\item\emph{co-Kähler} (see \cite{JV81}) or \emph{cosymplectic} (see \cite{Bla67}) if it is normal and almost cosymplectic, or equivalently if
\be
\lb(\nabla^g_X\vphi\rb)\lb(Y\rb)=0.
\ee
\item\emph{nearly cosymplectic} (see \cite{Bla71}) if
\be
\lb(\nabla^g_X\vphi\rb)\lb(X\rb)=0.
\ee
\item\emph{semi-cosymplectic} (see \cite{CG90}) if
\ba
\delta\Phi&=0, & \delta\eta&=0.
\ea
\item\emph{quasi-cosymplectic} (see \cite{CM92}) if
\be
\lb(\nabla^g_X\vphi\rb)\lb(Y\rb)+\lb(\nabla^g_{\vphi\lb(X\rb)}\vphi\rb)\lb(\vphi\lb(Y\rb)\rb)=\eta\lb(Y\rb)\cdot\nabla^g_{\vphi\lb(X\rb)}\xi.
\ee
\item\emph{almost $\alpha$-Kenmotsu} (see \cite{JV81}) if
\ba
d\Phi&=2\alpha\cdot\Phi\wedge\eta, & d\eta&=0
\ea
for $\alpha\in\R\backslash\lb\{0\rb\}$.
\item\emph{$\alpha$-Kenmotsu} (see \cite{JV81}) if it is normal and almost $\alpha$-Kenmotsu, or equivalently if
\be
\lb(\nabla^g_X\vphi\rb)\lb(Y\rb)=\alpha\lb(g\lb(\vphi\lb(X\rb),Y\rb)\cdot\xi-\eta\lb(Y\rb)\cdot\vphi\lb(X\rb)\rb)
\ee
for $\alpha\in\R\backslash\lb\{0\rb\}$.
\item\emph{almost Kenmotsu} (see \cites{Ken72,JV81}) if it is almost $1$-Kenmotsu.
\item\emph{Kenmotsu} (see \cites{Ken72,JV81}) if it is $1$-Kenmotsu.
\item\emph{almost $\alpha$-Sasaki} (see \cite{JV81}) if
\ba
d\Phi&=0, & d\eta&=2\alpha\cdot\Phi
\ea
for $\alpha\in\R\backslash\lb\{0\rb\}$.
\item\emph{$\alpha$-Sasaki} (see \cite{JV81}) if it is normal and almost $\alpha$-Sasaki, or equivalently if
\be
\lb(\nabla^g_X\vphi\rb)\lb(Y\rb)=\alpha\lb(g\lb(X,Y\rb)\cdot\xi-\eta\lb(Y\rb)\cdot X\rb)
\ee
for $\alpha\in\R\backslash\lb\{0\rb\}$.
\item\emph{almost Sasaki} (see \cite{JV81}) or \emph{contact metric} (see \cite{Bla76}) if it is almost $1$-Sasaki.
\item\emph{Sasaki} (see \cites{SH61,SH62,Bla76}) if it is $1$-Sasaki.
\item\emph{quasi-Sasaki} (see \cite{Bla67}) if it is normal and
\be
d\Phi=0.
\ee
\item\emph{nearly Sasaki} (see \cite{Bla76}) if
\be
\lb(\nabla^g_X\vphi\rb)\lb(X\rb)=g\lb(X,X\rb)\cdot\xi-\eta\lb(X\rb)\cdot X.
\ee
\item\emph{trans-Sasaki} (see \cite{Oub85}) if
\ba
4\cdot\lb(\nabla^g_X\Phi\rb)\lb(Y,Z\rb)&=g\lb(X,\delta\Phi\lb(\xi\rb)\cdot Z+\delta\eta\cdot\vphi\lb(Z\rb)\rb)\eta\lb(Y\rb)\\
&\phantom{=\,}-g\lb(X,\delta\Phi\lb(\xi\rb)\cdot Y+\delta\eta\cdot\vphi\lb(Y\rb)\rb)\eta\lb(Z\rb).
\ea
\item\emph{K-contact} (see \cite{Bla76}) if it is contact metric and $\xi$ is a Killing vector field with respect to $g$.
\eite
\begin{thm}\label{thm:2}
Let $\lb(M^5,g,\xi,\eta,\vphi\rb)$ be a $5$-dimensional almost contact metric manifold.
Then the following hold:
\btab{c|c}
If $\lb(M^5,g,\xi,\eta,\vphi\rb)$ is & then it is of class \\\hline
normal & $\W_1\op\W_3\op\W_5\op\W_6\op\W_8$\\\hline
almost cosymplectic & $\W_9\op\W_{10}$\\\hline
nearly cosymplectic & $\W_4\op\W_7$\\\hline
semi-cosymplectic & $\W_2\op\W_4\op\W_5\op\W_6\op\W_7\op\W_8\op\W_9\op\W_{10}$\\\hline
quasi-cosymplectic & $\W_4\op\W_7\op\W_9\op\W_{10}$\\\hline
normal and semi-cosymplectic & $\W_5\op\W_8$\\\hline
almost $\alpha$-Kenmotsu & $\W_1\op\W_9\op\W_{10}$\\\hline
almost $\alpha$-Sasaki & $\W_3\op\W_9\op\W_{10}$\\\hline
quasi-Sasaki & $\W_3\op\W_5$\\\hline
nearly Sasaki & $\W_3\op\W_4\op\W_7$\\\hline
trans-Sasaki & $\W_1\op\W_3$\\\hline
K-contact & $\W_3\op\W_9$
\etab
Moreover, $\lb(M^5,g,\xi,\eta,\vphi\rb)$ is cosymplectic if and only if the almost contact metric manifold is integrable. If $\lb(M^5,g,\xi,\eta,\vphi\rb)$ is almost $\alpha$-Kenmotsu, almost $\alpha$-Sasaki, nearly Sasaki or K-contact, then the almost contact metric manifold is not integrable.
\end{thm}
\begin{proof}
The results are a consequence of lemma \ref{lem:2}, theorem \ref{thm:1} and formula (\ref{eqn:1}). For example, suppose that $\lb(M^5,g,\xi,\eta,\vphi\rb)$ is quasi-Sasaki. Then, by theorem \ref{thm:1}, we have $\Gamma\in\W_1\op\W_3\op\W_5\op\W_6\op\W_8$. This, together with
\be
0=d\Phi\lb(X,Y,Z\rb)=\mfr{S}_{X,Y,Z}\lb(\nabla^{g}_{X}\Phi\rb)\lb(Y,Z\rb)
\ee
and formula (\ref{eqn:1}), leads to $\Gamma\in\W_3\op\W_5$.
\end{proof}\noindent
Using the same method, we are able to prove
\begin{thm}\label{thm:3}
Let $\lb(M^5,g,\xi,\eta,\vphi\rb)$ be a $5$-dimensional almost contact metric manifold.
Then the following hold:
\btab{c|c}
If $\lb(M^5,g,\xi,\eta,\vphi\rb)$ is of class & then it is \\\hline
$\W_1\op\W_3\op\W_5\op\W_6\op\W_8$ & normal\\\hline
$\W_9\op\W_{10}$ & quasi- and almost cosymplectic\\\hline
$\W_4\op\W_5\op\W_7\op\W_8\op\W_9\op\W_{10}$ & semi-cosymplectic\\\hline
$\W_3\op\W_5$ & quasi-Sasaki\\\hline
$\W_1\op\W_3$ & trans-Sasaki
\etab
\end{thm}\noindent
Moreover, we investigate intersections between certain types using this technique.
\begin{prop}
If a normal almost contact metric $5$-manifold $\lb(M^5,g,\xi,\eta,\vphi\rb)$ is K-contact or nearly Sasaki, then $\lb(M^5,g,\xi,\eta,\vphi\rb)$ is Sasaki. Moreover, there exists no non-integrable almost contact metric $5$-manifold that is
\bite
\item[a)] nearly cosymplectic and quasi-cosymplectic.
\item[b)] almost $\alpha$-Kenmotsu and semi-cosymplectic.
\item[c)] almost $\alpha$-Sasaki and semi-cosymplectic.
\item[d)] nearly Sasaki and semi-cosymplectic.
\item[e)] K-contact and semi-cosymplectic.
\eite
\end{prop}\noindent
Finally, we visualize some of the previous statements in figure \ref{pic:1}.
\begin{figure}[t]
\centering
%
%
\setlength{\unitlength}{3947sp}%
\begingroup\makeatletter\ifx\SetFigFont\undefined%
\gdef\SetFigFont#1#2#3#4#5{%
  \reset@font\fontsize{#1}{#2pt}%
  \fontfamily{#3}\fontseries{#4}\fontshape{#5}%
  \selectfont}%
\fi\endgroup%
\begin{picture}(5637,4517)(1097,-5292)
\thicklines
{\color[rgb]{0,0,0}\put(1331,-3492){\oval(360,360)[bl]}
\put(1331,-988){\oval(360,360)[tl]}
\put(3892,-3492){\oval(360,360)[br]}
\put(3892,-988){\oval(360,360)[tr]}
\put(1331,-3672){\line( 1, 0){2561}}
\put(1331,-808){\line( 1, 0){2561}}
\put(1151,-3492){\line( 0, 1){2504}}
\put(4072,-3492){\line( 0, 1){2504}}
}%
{\color[rgb]{0,0,0}\put(2874,-5079){\oval(360,360)[bl]}
\put(2874,-3130){\oval(360,360)[tl]}
\put(5370,-5079){\oval(360,360)[br]}
\put(5370,-3130){\oval(360,360)[tr]}
\put(2874,-5259){\line( 1, 0){2496}}
\put(2874,-2950){\line( 1, 0){2496}}
\put(2694,-5079){\line( 0, 1){1949}}
\put(5550,-5079){\line( 0, 1){1949}}
}%
\thinlines
{\color[rgb]{0,0,0}\put(2874,-4831){\oval(360,360)[bl]}
\put(2874,-4021){\oval(360,360)[tl]}
\put(5370,-4831){\oval(360,360)[br]}
\put(5370,-4021){\oval(360,360)[tr]}
\put(2874,-5011){\line( 1, 0){2496}}
\put(2874,-3841){\line( 1, 0){2496}}
\put(2694,-4831){\line( 0, 1){810}}
\put(5550,-4831){\line( 0, 1){810}}
}%
{\color[rgb]{0,0,0}\put(1299,-4403){\oval(360,360)[bl]}
\put(1299,-4366){\oval(360,360)[tl]}
\put(3579,-4403){\oval(360,360)[br]}
\put(3579,-4366){\oval(360,360)[tr]}
\put(1299,-4583){\line( 1, 0){2280}}
\put(1299,-4186){\line( 1, 0){2280}}
\put(1119,-4403){\line( 0, 1){ 37}}
\put(3759,-4403){\line( 0, 1){ 37}}
}%
{\color[rgb]{0,0,0}\put(4306,-4404){\oval(360,360)[bl]}
\put(4306,-4344){\oval(360,360)[tl]}
\put(6460,-4404){\oval(360,360)[br]}
\put(6460,-4344){\oval(360,360)[tr]}
\put(4306,-4584){\line( 1, 0){2154}}
\put(4306,-4164){\line( 1, 0){2154}}
\put(4126,-4404){\line( 0, 1){ 60}}
\put(6640,-4404){\line( 0, 1){ 60}}
}%
{\color[rgb]{0,0,0}\put(4433,-1530){\oval(360,360)[bl]}
\put(4433,-1439){\oval(360,360)[tl]}
\put(6498,-1530){\oval(360,360)[br]}
\put(6498,-1439){\oval(360,360)[tr]}
\put(4433,-1710){\line( 1, 0){2065}}
\put(4433,-1259){\line( 1, 0){2065}}
\put(4253,-1530){\line( 0, 1){ 91}}
\put(6678,-1530){\line( 0, 1){ 91}}
}%
{\color[rgb]{0,0,0}\put(4410,-2195){\oval(360,360)[bl]}
\put(4410,-2115){\oval(360,360)[tl]}
\put(6532,-2195){\oval(360,360)[br]}
\put(6532,-2115){\oval(360,360)[tr]}
\put(4410,-2375){\line( 1, 0){2122}}
\put(4410,-1935){\line( 1, 0){2122}}
\put(4230,-2195){\line( 0, 1){ 80}}
\put(6712,-2195){\line( 0, 1){ 80}}
}%
{\color[rgb]{0,0,0}\put(1591,-2605){\oval(360,360)[bl]}
\put(1591,-2565){\oval(360,360)[tl]}
\put(3594,-2605){\oval(360,360)[br]}
\put(3594,-2565){\oval(360,360)[tr]}
\put(1591,-2785){\line( 1, 0){2003}}
\put(1591,-2385){\line( 1, 0){2003}}
\put(1411,-2605){\line( 0, 1){ 40}}
\put(3774,-2605){\line( 0, 1){ 40}}
}%
{\color[rgb]{0,0,0}\put(1591,-2049){\oval(360,360)[bl]}
\put(1591,-1427){\oval(360,360)[tl]}
\put(3577,-2049){\oval(360,360)[br]}
\put(3577,-1427){\oval(360,360)[tr]}
\put(1591,-2229){\line( 1, 0){1986}}
\put(1591,-1247){\line( 1, 0){1986}}
\put(1411,-2049){\line( 0, 1){622}}
\put(3757,-2049){\line( 0, 1){622}}
}%
{\color[rgb]{0,0,0}\put(1681,-1925){\oval(360,360)[bl]}
\put(1681,-1833){\oval(360,360)[tl]}
\put(3419,-1925){\oval(360,360)[br]}
\put(3419,-1833){\oval(360,360)[tr]}
\put(1681,-2105){\line( 1, 0){1738}}
\put(1681,-1653){\line( 1, 0){1738}}
\put(1501,-1925){\line( 0, 1){ 92}}
\put(3599,-1925){\line( 0, 1){ 92}}
}%
{\color[rgb]{0,0,0}\put(4036,-4629){\oval(360,360)[bl]}
\put(4036,-3491){\oval(360,360)[tl]}
\put(5370,-4629){\oval(360,360)[br]}
\put(5370,-3491){\oval(360,360)[tr]}
\put(4036,-4809){\line( 1, 0){1334}}
\put(4036,-3311){\line( 1, 0){1334}}
\put(3856,-4629){\line( 0, 1){1138}}
\put(5550,-4629){\line( 0, 1){1138}}
}%
\put(1753,-1006){\makebox(0,0)[lb]{\smash{{\SetFigFont{12}{14.4}{\familydefault}{\mddefault}{\updefault}{\color[rgb]{0,0,0}semi-cosymplectic}%
}}}}
\put(1734,-1434){\makebox(0,0)[lb]{\smash{{\SetFigFont{12}{14.4}{\familydefault}{\mddefault}{\updefault}{\color[rgb]{0,0,0}quasi-cosymplectic}%
}}}}
\put(1648,-1845){\makebox(0,0)[lb]{\smash{{\SetFigFont{12}{14.4}{\familydefault}{\mddefault}{\updefault}{\color[rgb]{0,0,0}almost cosymplectic}%
}}}}
\put(1678,-2580){\makebox(0,0)[lb]{\smash{{\SetFigFont{12}{14.4}{\rmdefault}{\mddefault}{\updefault}{\color[rgb]{0,0,0}nearly cosymplectic}%
}}}}
\put(4548,-1453){\makebox(0,0)[lb]{\smash{{\SetFigFont{12}{14.4}{\rmdefault}{\mddefault}{\updefault}{\color[rgb]{0,0,0}K-contact non-Sasaki}%
}}}}
\put(4380,-2124){\makebox(0,0)[lb]{\smash{{\SetFigFont{12}{14.4}{\rmdefault}{\mddefault}{\updefault}{\color[rgb]{0,0,0}nearly Sasaki non-Sasaki}%
}}}}
\put(3855,-3149){\makebox(0,0)[lb]{\smash{{\SetFigFont{12}{14.4}{\rmdefault}{\mddefault}{\updefault}{\color[rgb]{0,0,0}normal}%
}}}}
\put(4175,-3493){\makebox(0,0)[lb]{\smash{{\SetFigFont{12}{14.4}{\familydefault}{\mddefault}{\updefault}{\color[rgb]{0,0,0}quasi-Sasaki}%
}}}}
\put(3520,-4024){\makebox(0,0)[lb]{\smash{{\SetFigFont{12}{14.4}{\familydefault}{\mddefault}{\updefault}{\color[rgb]{0,0,0}trans-Sasaki}%
}}}}
\put(4781,-4350){\makebox(0,0)[lb]{\smash{{\SetFigFont{12}{14.4}{\familydefault}{\mddefault}{\updefault}{\color[rgb]{0,0,0}almost $\alpha$-Sasaki}%
}}}}
\put(1670,-4365){\makebox(0,0)[lb]{\smash{{\SetFigFont{12}{14.4}{\rmdefault}{\mddefault}{\updefault}{\color[rgb]{0,0,0}almost $\alpha$-Kenmotsu}%
}}}}
\end{picture}%
\caption{Euler diagram of certain types of non-in\-te\-gra\-ble almost contact metric $5$-manifolds}
\label{pic:1}
\end{figure}
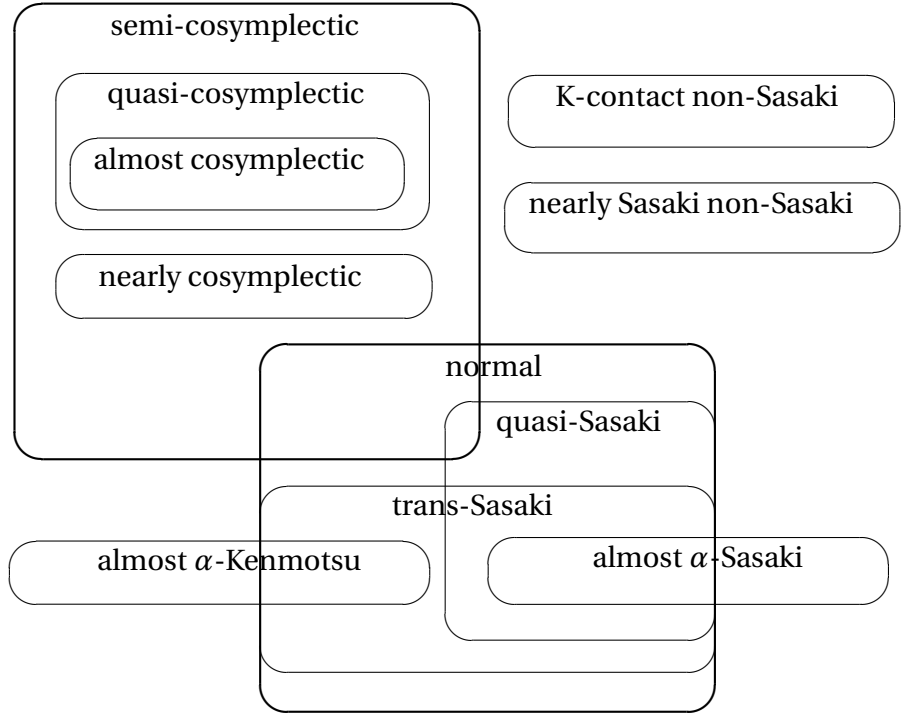
%
%
%
%
\section{Compatible connections}\label{sec:4}\noindent
Let $\nabla$ be a metric connection on $\lb(M^5,g,\xi,\eta,\vphi\rb)$, i.e.\
\be
g\lb(\nabla_XY,Z\rb)=g\lb(\nabla^{g}_XY,Z\rb)+A\lb(X,Y,Z\rb)
\ee
for $A\in\A$. Its torsion $T$, viewed as a $\lb(3,0\rb)$-tensor, is given by
\ba
T\lb(X,Y,Z\rb)&=g\lb(\nabla_XY-\nabla_YX-\lb[X,Y\rb],Z\rb)\\
&=A\lb(X,Y,Z\rb)-A\lb(Y,X,Z\rb).
\ea
Consequently, we have $T\in\T$. We say that $T$ is
\bite
\item \emph{vectorial} if $T\in\T_1$, or equivalently if $A\in\A_1$.
\item \emph{totally skew-symmetric} if $T\in\T_2$, or equivalently if $A\in\A_2$.
\item \emph{cyclic} if $T\in\T_1\op\T_3$, or equivalently if $A\in\A_1\op\A_3$.
\item \emph{traceless cyclic} if $T\in\T_3$, or equivalently if $A\in\A_3$.
\eite
The connection forms
\be
\omega_{ij}:=g\lb(\nabla e_i,e_j\rb)
\ee
of $\nabla$ define a $1$-form
\be
\Omega:=\lb(\omega_{ij}\rb)_{1\leq i,j\leq 5}
\ee
with values in the Lie algebra $\so(5)$,
\be
\Omega\lb(X\rb)=\Omega^g\lb(X\rb)+A\lb(X\rb).
\ee
We project onto $\m$:
\ba
\pr_{\m}\lb(\Omega\lb(X\rb)\rb)&=\Gamma\lb(X\rb)+\pr_{\m}\lb(A\lb(X\rb)\rb)\\
&=\Gamma\lb(X\rb)+\pr_{\W}\lb(A\rb)\lb(X\rb).
\ea
Therefore, $\nabla$ preserves the underlying almost contact metric structure, i.e.\
\ba
\nabla\xi&=0, &
\nabla\eta&=0, &
\nabla\vphi&=0
\ea
are satisfied, if and only if
\be\label{eqn:2}\tag{$\star\star$}
\Gamma+\pr_{\W}\lb(A\rb)=0.
\ee
In this case, we also say that the connection is \emph{compatible} with the almost contact metric structure. With a glance at lemma \ref{lem:1} and proposition \ref{prop:1}, we immediately have
\begin{prop}\label{prop:3}
Let $\lb(M^5,g,\xi,\eta,\vphi\rb)$ be a $5$-dimensional almost contact metric manifold equipped with a metric connection $\nabla^c$ compatible with the almost contact metric structure. If the torsion of $\nabla^c$ is
\bite
\item[a)] vectorial, then $\lb(M^5,g,\xi,\eta,\vphi\rb)$ is of class
\be
\W_1\op\W_2.
\ee
\item[b)] totally skew-symmetric, then $\lb(M^5,g,\xi,\eta,\vphi\rb)$ is of class
\be
\W_3\op\W_4\op\W_5\op\W_6.
\ee
\item[c)] traceless cyclic, then $\lb(M^5,g,\xi,\eta,\vphi\rb)$ is of class
\be
\W_2\op\W_3\op\W_5\op\W_6\op\W_7\op\W_8\op\W_9\op\W_{10}.
\ee
\eite
\end{prop}\noindent
Using our approach of \autoref{sec:3}, we characterize each of these classes in terms of differential equations.
\begin{prop}\label{prop:4}
An almost contact metric $5$-manifold is of class
\bite
\item[a)] $\W_1\op\W_2$ if and only if
\ba
N\lb(X,Y,Z\rb)&=\eta\lb(X\rb)d\eta\lb(Y,Z\rb), &
d\Phi&=-2\lb(\frac{1}{4}\,\delta\eta\cdot\eta+\xi\hook d\eta\rb)\wedge\Phi.
\ea
\item[b)] $\W_3\op\W_4\op\W_5\op\W_6$ if and only if
\ba
N\lb(X,Y,Z\rb)+N\lb(Z,Y,X\rb)&=0, &
d\Phi\lb(X,Y,\xi\rb)+d\Phi\lb(\vphi\lb(X\rb),\vphi\lb(Y\rb),\xi\rb)&=0.
\ea
\item[c)] $\W_2\op\W_3\op\W_5\op\W_6\op\W_7\op\W_8\op\W_9\op\W_{10}$ if and only if
\ba
\mfr{S}_{X,Y,Z}N\lb(X,Y,Z\rb)&=0, &
d\Phi\wedge\Phi&=0.
\ea
\eite
\end{prop}
We now solve (\ref{eqn:2}).
\begin{thm}\label{thm:6}
Let $\lb(M^5,g,\xi,\eta,\vphi\rb)$ be a $5$-dimensional almost contact metric manifold of class $\W_1\op\W_2$. Then there exists a unique metric connection $\nabla^c$ with vectorial torsion compatible with the almost contact metric structure. $\nabla^c$ is given by
\be
g\lb(\nabla^c_XY,Z\rb)=g\lb(\nabla^{g}_XY,Z\rb)-\theta_1\lb(\frac{1}{4}\,\delta\eta\cdot\eta+\xi\hook d\eta\rb)\lb(X,Y,Z\rb).
\ee
Moreover, $\lb(M^5,g,\xi,\eta,\vphi\rb)$ is of class
\bite
\item[a)] $\W_1$ if and only if $d\eta=0$.
\item[b)] $\W_2$ if and only if $\delta\eta=0$.
\eite
\end{thm}
\begin{proof}
A direct computation verifies that
\be
\Gamma=\lb(\pr_{\W}\circ\theta_1\rb)\lb(\frac{1}{4}\,\delta\eta\cdot\eta+\xi\hook d\eta\rb)
\ee
and
\be
d\eta\wedge\eta=0
\ee
hold if $\Gamma\in\W_1\op\W_2$. Moreover, the restriction of the projection map $\pr_\W$ to $\A_1$ is one-to-one.
\end{proof}\noindent
\begin{thm}\label{thm:4}
Let $\lb(M^5,g,\xi,\eta,\vphi\rb)$ be a $5$-dimensional almost contact metric manifold of class $\W_3\op\W_4\op\W_5\op\W_6$. Then there exists a unique metric connection $\nabla^c$ with totally skew-symmetric torsion compatible with the almost contact metric structure. $\nabla^c$ is given by
\be
g\lb(\nabla^c_XY,Z\rb)=g\lb(\nabla^{g}_XY,Z\rb)
+\frac{1}{2}\lb(d\eta\wedge\eta+\xi\hook\lb(\ast d\Phi\wedge\Phi\rb)\rb)\lb(X,Y,Z\rb).
\ee
Moreover, $\lb(M^5,g,\xi,\eta,\vphi\rb)$ is of class
\bite
\item[a)] $\W_3\op\W_4\op\W_5$ if and only if $\xi\hook\lb(\ast d\Phi\wedge\Phi\rb)=0$.
\item[b)] $\W_3\op\W_4\op\W_6$ if and only if $d\eta\wedge\eta=\ast d\eta$.
\item[c)] $\W_3\op\W_5\op\W_6$ if and only if $N=0$.
\item[d)] $\W_4\op\W_5\op\W_6$ if and only if $\xi\hook\delta\Phi=0$.
\eite
\end{thm}
\begin{proof}
Suppose that $\Gamma\in\W_3\op\W_4\op\W_5\op\W_6$. Then,
\be
-2\,\Gamma=\lb(\pr_{\W}\circ\theta_2\rb)\lb(d\eta\wedge\eta+\xi\hook\lb(\ast d\Phi\wedge\Phi\rb)\rb)
\ee
and
\ba
2\lb(d\eta\wedge\eta+\xi\hook\lb(\ast d\Phi\wedge\Phi\rb)\rb)=&\lb(\xi\hook\delta\Phi\rb)\cdot\underbrace{\Phi\wedge\eta}_{\in\Lambda^3_1}+\underbrace{\lb(d\eta\wedge\eta+\ast d\eta-\lb(\xi\hook\delta\Phi\rb)\cdot\Phi\wedge\eta\rb)}_{\in\Lambda^3_2}\\
&+\underbrace{\lb(d\eta\wedge\eta-\ast d\eta\rb)}_{\in\Lambda^3_3}+2\,\underbrace{\xi\hook\lb(\ast d\Phi\wedge\Phi\rb)}_{\in\Lambda^3_4}.
\ea
are satisfied. Moreover, the restriction of $\pr_\W$ to $\A_2$ is one-to-one. Theorems \ref{thm:2} and \ref{thm:3} complete the proof.
\end{proof}
\begin{rmk}
The first part of theorem \ref{thm:4}, together with proposition \ref{prop:3},b), yields the same result as theorem 8.2 of \cite{FI02}. Indeed, the Nijenhuis tensor is totally skew-symmetric and $\xi$ is a Killing vector field with respect to $g$ if and only if the almost contact metric $5$-manifold $\lb(M^5,g,\xi,\eta,\vphi\rb)$ is of class $\W_3\op\W_4\op\W_5\op\W_6$. Moreover, 
\ba
\lb(\ast d\Phi\wedge\Phi\rb)\lb(\xi,X,Y,Z\rb)=&-d\Phi\lb(\vphi\lb(X\rb),\vphi\lb(Y\rb),\vphi\lb(Z\rb)\rb)+N\lb(Z,X,Y\rb)\\&-\mfr{S}_{X,Y,Z}\eta\lb(X\rb)N\lb(\xi,Y,Z\rb)
\ea
holds in this case.
\end{rmk}\noindent
A lengthy but similar computation for the remaining case results in
\begin{thm}\label{thm:5}
Let $\lb(M^5,g,\xi,\eta,\vphi\rb)$ be a $5$-dimensional almost contact metric manifold of class $\W_2\op\W_3\op\W_5\op\W_6\op\W_7\op\W_8\op\W_9\op\W_{10}$. Then there exists a unique metric connection $\nabla^c$ with traceless cyclic torsion compatible with the almost contact metric structure. $\nabla^c$ is given by
\ba
g\lb(\nabla^c_XY,Z\rb)=g\lb(\nabla^{g}_XY,Z\rb)&
-\frac{1}{2}\,\theta_3\lb(d\eta+\lb(\xi\hook d\eta\rb)\wedge\eta\rb)\lb(X,Y,Z\rb)\\
&+\frac{1}{4}\,\theta_4\lb(\ast\lb(\delta\Phi\wedge\Phi\wedge\eta\rb)-3\lb(\xi\hook d\eta\rb)\rb)\lb(X,Y,Z\rb)\\
&+\frac{1}{3}\,\theta_5\lb(\xi\hook d\eta\rb)\lb(X,Y,Z\rb)+\frac{1}{2}\,\theta_3\lb(\ast d\Phi\rb)\lb(\vphi\lb(X\rb),Y,Z\rb)\\
&-\frac{1}{4}\,N\lb(\vphi^2\lb(X\rb),\vphi\lb(Y\rb),\vphi\lb(Z\rb)\rb)+\frac{1}{2}\,\eta\lb(Y\rb)N\lb(\vphi\lb(X\rb),\xi,\vphi\lb(Z\rb)\rb)\\
&+\frac{1}{2}\,\eta\lb(Z\rb)N\lb(\vphi\lb(X\rb),\vphi\lb(Y\rb),\xi\rb).
\ea
Moreover, $\lb(M^5,g,\xi,\eta,\vphi\rb)$ is of class
\bite
\item[a)] $\W_2\op\W_3\op\W_5\op\W_6\op\W_7\op\W_8\op\W_9$ if and only if
\be
N\lb(\vphi\lb(X\rb),\vphi\lb(Y\rb),\xi\rb)+N\lb(\vphi\lb(Y\rb),\vphi\lb(X\rb),\xi\rb)=0.
\ee
\item[b)] $\W_2\op\W_3\op\W_5\op\W_6\op\W_7\op\W_8\op\W_{10}$ if and only if
\be
N\lb(\vphi\lb(X\rb),\vphi\lb(Y\rb),\vphi\lb(Z\rb)\rb)=0.
\ee
\item[c)] $\W_2\op\W_3\op\W_5\op\W_6\op\W_7\op\W_9\op\W_{10}$ if and only if 
\be
\lb(\ast d\Phi\rb)\lb(\vphi\lb(X\rb),\vphi^2\lb(Y\rb)\rb)+\lb(\ast d\Phi\rb)\lb(\vphi\lb(Y\rb),\vphi^2\lb(X\rb)\rb)=0.
\ee
\item[d)] $\W_2\op\W_3\op\W_5\op\W_6\op\W_8\op\W_9\op\W_{10}$ if and only if 
\be
N\lb(\xi,\vphi\lb(X\rb),\vphi\lb(Y\rb)\rb)=0.
\ee
\item[e)] $\W_2\op\W_3\op\W_5\op\W_7\op\W_8\op\W_9\op\W_{10}$ if and only if $\ast\lb(\delta\Phi\wedge\Phi\wedge\eta\rb)-3\lb(\xi\hook d\eta\rb)=0$.
\item[f)] $\W_2\op\W_3\op\W_6\op\W_7\op\W_8\op\W_9\op\W_{10}$ if and only if $d\eta+\lb(\xi\hook d\eta\rb)\wedge\eta=\ast\lb(d\eta\wedge\eta\rb)$.
\item[g)] $\W_2\op\W_5\op\W_6\op\W_7\op\W_8\op\W_9\op\W_{10}$ if and only if $\xi\hook \delta\Phi=0$.
\item[h)] $\W_3\op\W_5\op\W_6\op\W_7\op\W_8\op\W_9\op\W_{10}$ if and only if $\xi\hook d\eta=0$.
\eite
\end{thm}
%
%
%
%
\section{Examples}\label{sec:5}\noindent
Throughout this section, the curvature tensor of a metric connection $\nabla$, viewed as $(4,0)$-tensor, is defined by
\be
\mrm{R}\left(X,Y,Z,V\right)=g\left(\nabla_X\nabla_YZ-\nabla_Y\nabla_XZ-\nabla_{\left[X,Y\right]}Z,V\right).
\ee
%
%
\subsection{Examples of class \texorpdfstring{$\W_1$ and $\W_2$}{W_1 and W_2}}
Let $M^5$ be the Lie group
\be
\lb\{
\begin{pmatrix}
e^{-x_5}&0&0&0&x_1\\
0&e^{-x_5}&0&0&x_2\\
0&0&e^{-x_5}&0&x_3\\
0&0&0&e^{-x_5}&x_4\\
0&0&0&0&1
\end{pmatrix}\in\GL\lb(5,\R\rb)\setsep x_1,\ldots,x_5\in\R\rb\}
\ee
equipped with the left-invariant Riemannian metric
\be
g=e^{2x_5}\lb(dx_1^2+dx_2^2+dx_3^2+dx_4^2\rb)+dx_5^2.
\ee
$\lb(M^5,g\rb)$ can be considered as the warped product $\R\x_f\R^4$ with warping function $f:\R\ni t\mapsto e^t\in\R$ (cf.\ \cite{BO69}). The vector fields
\ba
&e^{-x_5}\frac{\del}{\del x_1}, &
&e^{-x_5}\frac{\del}{\del x_2}, &
&e^{-x_5}\frac{\del}{\del x_3}, &
&e^{-x_5}\frac{\del}{\del x_4}, &
&\frac{\del}{\del x_5} &
\ea
are left-invariant. In the following discussion, we consider two almost contact metric structures on $\lb(M^5,g\rb)$.
\subsubsection{Class \texorpdfstring{$\W_1$}{W_1}}
The standard Kenmotsu structure $\lb(\xi,\eta,\vphi\rb)$ on $\lb(M^5,g\rb)$ (cf. \cite{Ken72}) is characterized by
\ba
\xi&=\frac{\del}{\del x_5}, &
\eta&=dx_5, &
\Phi&=e^{2x_5}\lb(dx_1\wedge dx_2+dx_3\wedge dx_4\rb).
\ea
Consequently, $\lb(e_1,e_2,e_3,e_4,e_5\rb)$ defined via
\ba
e_1&:=e^{-x_5}\frac{\del}{\del x_1},&
e_2&:=e^{-x_5}\frac{\del}{\del x_2},&
e_3&:=e^{-x_5}\frac{\del}{\del x_3},&
e_4&:=e^{-x_5}\frac{\del}{\del x_4},&
e_5&:=\frac{\del}{\del x_5}
\ea
is an adapted frame of $\lb(M^5,g,\xi,\eta,\vphi\rb)$. The non-zero connection forms of $\nabla^g$ with respect to this frame are
\ba
\omega_{15}^g&=-e_1, &
\omega_{25}^g&=-e_2, &
\omega_{35}^g&=-e_3, &
\omega_{45}^g&=-e_4.
\ea
Therefore,
\be
\Gamma = -e_1\ox e_{15}-e_2\ox e_{25}-e_3\ox e_{35}-e_4\ox e_{45}
\ee
is the intrinsic torsion of $\lb(M^5,g,\xi,\eta,\vphi\rb)$.
\begin{prop}\label{prop:9}
The almost contact metric manifold $\lb(M^5,g,\xi,\eta,\vphi\rb)$ has the following properties:
\bite
\item[a)]The almost contact metric manifold is of class $\W_1$.
\item[b)]The Nijenhuis tensor $N$ vanishes.
\item[c)]The fundamental form $\Phi$ and the $1$-form $\eta$ satisfy
\ba
d\Phi&=2\,\Phi\wedge\eta,&
\delta\Phi&=0,&
d\eta&=0.
\ea
\item[d)]The Riemannian curvature tensor $\mrm{R}^g$ is the identity map of $\Lambda^2$.
\eite
\end{prop}\noindent
As a consequence of proposition \ref{prop:9},a) and theorem \ref{thm:6}, there exists a unique metric connection $\nabla^c$ with vectorial torsion compatible with the almost contact metric structure. The torsion tensor of $\nabla^c$ is
\be
T^c=-e_{15}\ox e_1-e_{25}\ox e_2-e_{35}\ox e_3-e_{45}\ox e_4\in\T_{1,1}.
\ee
\begin{prop}
The metric connection $\nabla^c$ is flat. Moreover, its torsion tensor $T^c$ is parallel with respect to $\nabla^c$, i.e.\
\be
\nabla^cT^c=0.
\ee
\end{prop}
\subsubsection{Class \texorpdfstring{$\W_2$}{W_2}}
There exists an almost contact metric structure $\lb(\xi,\eta,\vphi\rb)$ on $\lb(M^5,g\rb)$ such that $\lb(e_1,e_2,e_3,e_4,e_5\rb)$ defined by
\ba
e_1&:=-\frac{\del}{\del x_5},&
e_2&:=e^{-x_5}\frac{\del}{\del x_2},&
e_3&:=e^{-x_5}\frac{\del}{\del x_3},&
e_4&:=e^{-x_5}\frac{\del}{\del x_4},&
e_5&:=e^{-x_5}\frac{\del}{\del x_1}
\ea
is an adapted frame of $\lb(M^5,g,\xi,\eta,\vphi\rb)$. With respect to this frame, the non-zero connection forms of the Levi-Civita connection are
\ba
\omega_{12}^g&=-e_2, &
\omega_{13}^g&=-e_3, &
\omega_{14}^g&=-e_4, &
\omega_{15}^g&=-e_5.
\ea
Therefore, the intrinsic torsion of $\lb(M^5,g,\xi,\eta,\vphi\rb)$ is
\be
\Gamma =
-\frac{1}{2}e_3\ox \lb(e_{13}-e_{24}\rb)
-\frac{1}{2}e_4\ox \lb(e_{14}+e_{23}\rb)
-e_5\ox e_{15}.
\ee
\begin{prop}
The almost contact metric manifold $\lb(M^5,g,\xi,\eta,\vphi\rb)$ is of class $\W_2$. Moreover, the fundamental form $\Phi$ and the $1$-form $\eta$ satisfy
\ba
d\Phi&=-2\lb(\xi\hook d\eta\rb)\wedge\Phi,&
\delta\eta&=0.
\ea
\end{prop}\noindent
Theorems \ref{thm:6} and \ref{thm:5} now imply that there exist both a unique metric connection $\nabla^{c,1}$ with vectorial torsion and a unique metric connection $\nabla^{c,2}$ with traceless cyclic torsion each compatible with the almost contact metric
structure. Explicitly, the torsion tensors of $\nabla^{c,1}$ and $\nabla^{c,2}$ are
\ba
T^{c,1}&=-e_{12}\ox e_2-e_{13}\ox e_3-e_{14}\ox e_4-e_{15}\ox e_5\in\T_{1,2},\\
T^{c,2}&=\frac{1}{3}\lb(\lb(5\,e_{12}+4\,e_{34}\rb)\ox e_2-\lb(e_{13}-2\,e_{24}\rb)\ox e_3-\lb(e_{14}+2\,e_{23}\rb)\ox e_4\rb)-e_{15}\ox e_5\in\T_{3,5}.
\ea
\begin{prop}
The metric connections $\nabla^{c,1}$ and $\nabla^{c,2}$ have the following properties:
\bite
\item[a)]$\nabla^{c,1}$ is flat.
\item[b)]The torsion tensor of $\nabla^{c,1}$ is parallel with respect to $\nabla^{c,1}$.
\item[c)]The curvature tensor with respect to $\nabla^{c,2}$ is
\ba
\mrm{R}^{c,2}=&\frac{4}{3}e_{12}\ox\lb(2\,e_{12}+e_{34}\rb)+\frac{2}{9}\lb(3\,e_{13}+4\,e_{24}\rb)\ox\lb(e_{13}+e_{24}\rb)\\
&-\frac{8}{9}e_{34}\ox\lb(e_{12}-e_{34}\rb)+\frac{2}{9}\lb(3\,e_{14}-4\,e_{23}\rb)\ox\lb(e_{14}-e_{23}\rb).
\ea
\item[d)]The Ricci tensor with respect to $\nabla^{c,2}$ is
\be
\Ric^{c,2}=\diag\lb(-4,-\frac{40}{9},-\frac{22}{9},-\frac{22}{9},0\rb).
\ee
\item[e)]The holonomy algebra of $\nabla^{c,2}$ is
\be
\hol^{c,2}=\un\lb(2\rb).
\ee
\eite
\end{prop}\noindent
Consequently, $T^{c,2}$ is not parallel with respect to $\nabla^{c,2}$.
%
%
%
\subsection{Examples of class \texorpdfstring{$\W_3$ and $\W_5$}{W_3 and W_5}}
Let $H$ be the $5$-dimensional Heisenberg group,
\be
H=\lb\{
\begin{pmatrix}
1&x_1&x_2&x_5\\
0&1&0&x_3\\
0&0&1&x_4\\
0&0&0&1
\end{pmatrix}\in\GL\lb(4,\R\rb)\setsep x_1,\ldots,x_5\in\R\rb\},
\ee
endowed with the left-invariant Riemannian metric
\be
g=\frac{1}{4}\lb(dx_1^2+dx_2^2+dx_3^2+dx_4^2+\lb(dx_5-x_1dx_3-x_2dx_4\rb)^2\rb).
\ee
The vector fields
\ba
&2\frac{\del}{\del x_1}, &
&2\frac{\del}{\del x_2}, &
&2\lb(\frac{\del}{\del x_3}+x_1\frac{\del}{\del x_5}\rb), &
&2\lb(\frac{\del}{\del x_4}+x_2\frac{\del}{\del x_5}\rb), &
&2\frac{\del}{\del x_5} &
\ea
are left-invariant and dual to
\ba
&\frac{1}{2}dx_1, &
&\frac{1}{2}dx_2, &
&\frac{1}{2}dx_3, &
&\frac{1}{2}dx_4, &
&\frac{1}{2}\lb(dx_5-x_1dx_3-x_2dx_4\rb).
\ea
There exist two almost contact metric structures on $\lb(H,g\rb)$. As before, we discuss these structures separately.
\subsubsection{Class \texorpdfstring{$\W_3$}{W_3}}
Let $\lb(\xi,\eta,\vphi\rb)$ be the standard Sasakian structure on $\lb(H,g\rb)$ (cf.\ \cite{Bla76}), i.e.\
\ba
\xi&=2\frac{\del}{\del x_5}, &
\eta&=\frac{1}{2}\lb(dx_5-x_1dx_3-x_2dx_4\rb), &
\Phi&=-\frac{1}{4}\lb(dx_1\wedge dx_3+dx_2\wedge dx_4\rb).
\ea
Then, $\lb(e_1,e_2,e_3,e_4,e_5\rb)$ with
\ba
e_1&:=2\lb(\frac{\del}{\del x_3}+x_1\frac{\del}{\del x_5}\rb),&
e_2&:=2\frac{\del}{\del x_1},&&\\
e_3&:=2\lb(\frac{\del}{\del x_4}+x_2\frac{\del}{\del x_5}\rb),&
e_4&:=2\frac{\del}{\del x_2},&
e_5&:=2\frac{\del}{\del x_5}
\ea
is an adapted frame of $\lb(H,g,\xi,\eta,\vphi\rb)$. With respect to this frame, the non-zero
connection forms of the Levi-Civita connection are
\ba
\omega_{12}^g&=\omega_{34}^g=e_5, &
\omega_{15}^g&=e_2, &
\omega_{25}^g&=-e_1, &
\omega_{35}^g&=e_4, &
\omega_{45}^g&=-e_3.
\ea
Consequently,
\be
\Gamma = -e_1\ox e_{25}+e_2\ox e_{15}-e_3\ox e_{45}+e_4\ox e_{35}.
\ee
\begin{prop}
The almost contact metric manifold $\lb(H,g,\xi,\eta,\vphi\rb)$ has the following properties:
\bite
\item[a)]The almost contact metric manifold is of class $\W_3$.
\item[b)]The Nijenhuis tensor $N$ vanishes.
\item[c)]The fundamental form $\Phi$ and the $1$-form $\eta$ satisfy
\ba
d\Phi&=0,&
\delta\Phi&=4\,\eta,&
d\eta&=2\,\Phi, &
\delta\eta&=0.
\ea
\item[d)]The Riemannian curvature tensor is
\ba
\mrm{R}^g=&3\lb(e_{12}\ox e_{12}+e_{34}\ox e_{34}\rb)+2\lb(e_{12}\ox e_{34}+e_{34}\ox e_{12}\rb)\\
          &+e_{13}\ox e_{24}+e_{24}\ox e_{13}-e_{14}\ox e_{23}-e_{23}\ox e_{14}\\
          &-e_{15}\ox e_{15}-e_{25}\ox e_{25}-e_{35}\ox e_{35}-e_{45}\ox e_{45}.
\ea
\item[e)]The Riemannian Ricci tensor is
\be
\Ric^g=\diag\lb(-2,-2,-2,-2,4\rb).
\ee
\item[f)]The Riemannian holonomy algebra is
\be
\hol^g=\so\lb(5\rb).
\ee
\eite
\end{prop}\noindent
With the aid of theorem \ref{thm:4}, we deduce that there exists a unique metric connection $\nabla^{c,1}$ with totally skew-symmetric torsion compatible with the underlying almost contact metric structure. Moreover, there exists a unique compatible connection $\nabla^{c,2}$ with traceless cyclic torsion (see theorem \ref{thm:5}). Explicitly, the corresponding torsion tensors are
\ba
T^{c,1}&=2\lb(e_{25}\ox e_1-e_{15}\ox e_2+e_{45}\ox e_3-e_{35}\ox e_4+\lb(e_{12}+e_{34}\rb)\ox e_5\rb)\in\T_{2,1},\\
T^{c,2}&=-e_{25}\ox e_1+e_{15}\ox e_2-e_{45}\ox e_3+e_{35}\ox e_4+2\lb(e_{12}+e_{34}\rb)\ox e_5\in\T_{3,1}.
\ea
\begin{prop}\label{prop:7}
The metric connections $\nabla^{c,1}$ and $\nabla^{c,2}$ have the following properties:
\bite
\item[a)]The curvature tensors with respect to $\nabla^{c,1}$ and $\nabla^{c,2}$ are
\ba
\mrm{R}^{c,1}=&4\lb(e_{12}+e_{34}\rb)\ox\lb(e_{12}+e_{34}\rb), &
\mrm{R}^{c,2}=&-2\lb(e_{12}+e_{34}\rb)\ox\lb(e_{12}+e_{34}\rb).
\ea
\item[b)]The Ricci tensors with respect to $\nabla^{c,1}$ and $\nabla^{c,2}$ are
\ba
\Ric^{c,1}&=\diag\lb(-4,-4,-4,-4,0\rb),&
\Ric^{c,2}&=\diag\lb(2,2,2,2,0\rb).
\ea
\item[c)]The holonomy algebrae of $\nabla^{c,1}$ and $\nabla^{c,2}$ are
\be
\hol^{c,1}=\hol^{c,2}=\un\lb(1\rb)\subset\un\lb(2\rb).
\ee
\item[d)]The torsion tensors of $\nabla^{c,1}$ and $\nabla^{c,2}$ are parallel, i.e.\ $\nabla^{c,i}T^{c,i}=0$.
\eite
\end{prop}
\subsubsection{Class \texorpdfstring{$\W_5$}{W_5}}
Let $\lb({\xi},{\eta},{\vphi}\rb)$ be the almost contact metric structure on $\lb(H,g\rb)$ such that $\lb({e}_1,{e}_2,{e}_3,{e}_4,{e}_5\rb)$ defined via
\ba
{e}_1&:=2\frac{\del}{\del x_1},&
{e}_2&:=2\frac{\del}{\del x_2},&&\\
{e}_3&:=2\lb(\frac{\del}{\del x_3}+x_1\frac{\del}{\del x_5}\rb),&
{e}_4&:=2\lb(\frac{\del}{\del x_4}+x_2\frac{\del}{\del x_5}\rb),&
{e}_5&:=2\frac{\del}{\del x_5}
\ea
is an adapted frame of $\lb(H,g,{\xi},{\eta},{\vphi}\rb)$. By a comparison with the adapted frame discussed before, we have
\ba
\omega_{13}^g&=\omega_{24}^g=-e_5&
\omega_{15}^g&=-e_3, &
\omega_{25}^g&=-e_4, &
\omega_{35}^g&=e_1, &
\omega_{45}^g&=e_2
\ea
and
\be
\omega_{12}^g=\omega_{14}^g=\omega_{23}^g=\omega_{34}^g=0.
\ee
Therefore,
\be
\Gamma = {e}_1\ox {e}_{35}+{e}_2\ox {e}_{45}-{e}_3\ox {e}_{15}-{e}_4\ox {e}_{25}.
\ee
\begin{prop}
The almost contact metric manifold $\lb(H,g,{\xi},{\eta},{\vphi}\rb)$ is of class $\W_5$. Moreover, the Nijenhuis tensor of $\lb(H,g,{\xi},{\eta},{\vphi}\rb)$ vanishes and
\ba
d{\Phi}&=0,&
\delta{\Phi}&=0,&
d{\eta}\wedge{\Phi}&=0, &
\delta{\eta}&=0.
\ea
\end{prop}\noindent
With respect to the given frame, the Riemannian curvature tensor is
\ba
\mrm{R}^g=&3\lb(e_{13}\ox e_{13}+e_{24}\ox e_{24}\rb)+2\lb(e_{13}\ox e_{24}+e_{24}\ox e_{13}\rb)\\
          &+e_{12}\ox e_{34}+e_{34}\ox e_{12}+e_{14}\ox e_{23}+e_{23}\ox e_{14}\\
          &-e_{15}\ox e_{15}-e_{25}\ox e_{25}-e_{35}\ox e_{35}-e_{45}\ox e_{45}.
\ea
Using the same arguments as before, there exist two uniquely determined metric connections ${\nabla}^{c,1}$, ${\nabla}^{c,2}$ compatible with the almost contact metric structure. The respective torsion tensors are
\ba
{T}^{c,1}&=-2\lb({e}_{35}\ox {e}_1+{e}_{45}\ox {e}_2-{e}_{15}\ox {e}_3-{e}_{25}\ox {e}_4+\lb({e}_{13}+{e}_{24}\rb)\ox {e}_5\rb)\in\T_{2,3},\\
{T}^{c,2}&={e}_{35}\ox {e}_1+{e}_{45}\ox {e}_2-{e}_{15}\ox {e}_3-{e}_{25}\ox {e}_4-2\lb({e}_{13}+{e}_{24}\rb)\ox {e}_5\in\T_{3,3}.
\ea
Again, we compute the corresponding curvature tensors:
\ba
\mrm{R}^{c,1}=&4\lb(e_{13}+e_{24}\rb)\ox\lb(e_{13}+e_{24}\rb), &
\mrm{R}^{c,2}=&-2\lb(e_{13}+e_{24}\rb)\ox\lb(e_{13}+e_{24}\rb).
\ea
Consequently, proposition \ref{prop:7},b)-d) is also valid for the connections ${\nabla}^{c,1}$ and ${\nabla}^{c,2}$ considered here.
%
%
%
\subsection{Examples of class \texorpdfstring{$\W_6$ and $\W_9$}{W_6 and W_9}}
The examples presented in this subsection are products of certain almost Hermitian $4$-manifolds with $\R$. The general construction scheme is as follows: Let $\lb(M^4,\tilde{g},J\rb)$ be a $4$-dimensional almost Hermitian manifold, i.e.\ $\lb(M^4,\tilde{g}\rb)$ is a $4$-dimensional Riemannian manifold equipped with an orthogonal almost complex structure $J: TM^4\ra TM^4$,
\ba
J^2 &= -\Id, & 
\tilde{g}(JX,JY) &= \tilde{g}(X,Y).
\ea
Moreover, let $t$ be the coordinate of $\R$. Then, on $M^4\x\R$, we set
\be\label{eqn:3}\tag{\decosix}\begin{aligned}
\xi&:=\frac{\del}{\del t},&
\eta&:=dt,&
\vphi\lb(X+f\frac{\del}{\del t}\rb)&:=JX\\
\end{aligned}\ee
and
\be\label{eqn:4}\tag{\decosix\decosix}
g\lb(X_1+f_1\frac{\del}{\del t},X_2+f_2\frac{\del}{\del t}\rb):=\tilde{g}\lb(X_1,X_2\rb)+f_1f_2
\ee
for $X,X_1,X_2$ tangent to $M^4$ and functions $f,f_1,f_2$ on $M^4\x\R$. As a result, the tuple $\lb(M^4\x\R,g,\xi,\eta,\vphi\rb)$ is an almost contact metric $5$-manifold (see \cite{Bla76}).
\subsubsection{Class \texorpdfstring{$\W_6$}{W_6}}
The following Riemannian $4$-manifold appears in the classification \cite{KT87}. Let $M^4$ be the Lie group
\be
\lb\{
\begin{pmatrix}
e^{x_4}&0&0&x_1\\
0&e^{x_4}&0&x_2\\
0&0&e^{-2x_4}&x_3\\
0&0&0&1
\end{pmatrix}\in\GL\lb(4,\R\rb)\setsep x_1,\ldots,x_4\in\R\rb\}
\ee
equipped with the left-invariant Riemannian metric
\be
\tilde{g}=e^{-2x_4}dx_1^2+e^{-2x_4}dx_2^2+e^{4x_4}dx_3^2+dx_4^2.
\ee
An orthonormal frame on $\lb(M^4,\tilde{g}\rb)$ is given by the left-invariant vector fields
\ba
e_1&:=e^{x_4}\frac{\del}{\del x_1}, &
e_2&:=e^{x_4}\frac{\del}{\del x_2}, &
e_3&:=e^{-2x_4}\frac{\del}{\del x_3}, &
e_4&:=\frac{\del}{\del x_4}.
\ea
Using these, we define an orthogonal almost complex structure $J$ on $\lb(M^4,\tilde{g}\rb)$ as
\ba
Je_1&=-e_2, &
Je_2&=e_1, &
Je_3&=-e_4, &
Je_4&=e_3.
\ea
Now, let $\lb(M^4\x\R,g,\xi,\eta,\vphi\rb)$ be the almost contact metric manifold constructed via (\ref{eqn:3}) and (\ref{eqn:4}). The connection forms of the Levi-Civita connection $\nabla^g$ with respect to the adapted frame $\lb(e_1,\ldots,e_4,e_5:=\xi\rb)$ are
\ba
\omega^g_{12}&=\omega^g_{13}=0, &
\omega^g_{14}&=e_1, &
\omega^g_{23}&=0, &
\omega^g_{24}&=e_2, &
\omega^g_{34}&=-2\,e_3, &
\omega^g_{i5}&=0.
\ea
Therefore, the intrinsic torsion of $\lb(M^4\x\R,g,\xi,\eta,\vphi\rb)$ is
\be
\Gamma=\frac{1}{2}\lb(  e_1\ox\lb(e_{14}+e_{23}\rb) - e_2\ox\lb(e_{13}-e_{24}\rb) \rb).
\ee
\begin{prop}
The almost contact metric manifold $\lb(M^4\x\R,g,\xi,\eta,\vphi\rb)$ has the following properties:
\bite
\item[a)]The almost contact metric manifold is of class $\W_6$.
\item[b)]The Nijenhuis tensor $N$ vanishes.
\item[c)]The fundamental form $\Phi$ and the $1$-form $\eta$ satisfy
\ba
d\Phi\wedge\delta\Phi&=-2\,\Phi\wedge\Phi,&
d\eta&=0, &
\delta\eta&=0.
\ea
\item[d)]The Riemannian curvature tensor is
\ba
\mrm{R}^g=&e_{12}\ox e_{12}-2\,e_{13}\ox e_{13}+e_{14}\ox e_{14}\\
          &-2\,e_{23}\ox e_{23}+e_{24}\ox e_{24}+4\, e_{34}\ox e_{34}.
\ea
\item[e)]The Riemannian Ricci tensor is
\be
\Ric^g=\diag\lb(0,0,0,-6,0\rb).
\ee
\item[f)]The Riemannian holonomy algebra is
\be
\hol^g=\so\lb(4\rb)\subset\so\lb(5\rb).
\ee
\eite
\end{prop}\noindent
Since $\lb(M^4\x\R,g,\xi,\eta,\vphi\rb)$ is of class $\W_6$, there exist both a unique metric connection $\nabla^{c,1}$ with totally skew-symmetric torsion and a unique metric connection $\nabla^{c,2}$ with traceless cyclic torsion each compatible with the underlying almost contact metric structure (cf.\ theorems \ref{thm:4} and \ref{thm:5}). The corresponding torsion tensors are
\ba
T^{c,1}&=-2\lb(e_{23}\ox e_1-e_{13}\ox e_2+e_{12}\ox e_3\rb)\in\T_{2,4},\\
T^{c,2}&=\frac{1}{2}\lb(
 \lb(e_{14}+e_{23}\rb)\ox e_1
-\lb(e_{13}-e_{24}\rb)\ox e_2\rb)
-\lb(e_{12}+e_{34}\rb)\ox e_3\in\T_{3,4}.
\ea
\begin{prop}
The metric connections $\nabla^{c,1}$ and $\nabla^{c,2}$ have the following properties:
\bite
\item[a)]The curvature tensors with respect to $\nabla^{c,1}$ and $\nabla^{c,2}$ are
\ba
\mrm{R}^{c,1}=&2\,\lb(e_{12}+e_{34}\rb)\ox\lb(e_{12}-e_{34}\rb)+\lb(e_{14}+e_{23}\rb)\ox\lb(e_{14}-e_{23}\rb)\\
&-\lb(e_{13}-e_{24}\rb)\ox\lb(e_{13}+e_{24}\rb)+6\,e_{34}\ox e_{34},\\
\mrm{R}^{c,2}=&\frac{1}{2}\lb(\lb(e_{12}-4\,e_{34}\rb)\ox\lb(e_{12}-e_{34}\rb)+\lb(e_{14}+2\,e_{23}\rb)\ox\lb(e_{14}-e_{23}\rb)\rb.\\
&\lb.-\lb(2\,e_{13}-e_{24}\rb)\ox\lb(e_{13}+e_{24}\rb)\rb).
\ea
\item[b)]The Ricci tensors with respect to $\nabla^{c,1}$ and $\nabla^{c,2}$ are
\ba
\Ric^{c,1}&=\diag\lb(-2,-2,-2,-6,0\rb),\\
\Ric^{c,2}&=\diag\lb(0,0,0,-3,0\rb).
\ea
\item[c)]The holonomy algebrae of $\nabla^{c,1}$ and $\nabla^{c,2}$ are
\ba
\hol^{c,1}&=\un\lb(2\rb),&
\hol^{c,2}&=\su\lb(2\rb)\subset\un\lb(2\rb).
\ea
\eite
\end{prop}\noindent
Consequently, $T^{c,i}$ is not parallel with respect to $\nabla^{c,i}$.
\subsubsection{Class \texorpdfstring{$\W_9$}{W_9}}
Let $M^4$ be the direct product of the Heisenberg group and $S^1$,
\be
M^4=\lb\{
\begin{pmatrix}
e^{2\pi \mrm{i} x_1}&0&0&0\\
0&1&x_2&x_4\\
0&0&1&x_3\\
0&0&0&1
\end{pmatrix}\in\GL\lb(4,\C\rb)\setsep x_1,\ldots,x_4\in\R\rb\},
\ee
endowed with the left-invariant Riemannian metric
\be
\tilde{g}=dx_1^2+dx_2^2+dx_3^2+\lb(dx_4-x_2dx_3\rb)^2.
\ee
The left-invariant vector fields
\ba
e_1&:=\frac{\del}{\del x_1}, &
e_2&:=\frac{\del}{\del x_2}, &
e_3&:=\frac{\del}{\del x_3}+x_2\frac{\del}{\del x_4}, &
e_4&:=\frac{\del}{\del x_4}
\ea
are dual to
\ba
&dx_1, &
&dx_2, &
&dx_3, &
&dx_4-x_2dx_3
\ea
and form an orthonormal frame on $\lb(M^4,\tilde{g}\rb)$. According to \cite{Abb84},
$\lb(M^4,\tilde{g}\rb)$ carries an orthogonal almost complex structure $J:TM^4\ra TM^4$,
\ba
Je_1&=-e_2, &
Je_2&=e_1, &
Je_3&=-e_4, &
Je_4&=e_3,
\ea
such that the Kähler form $\omega$ of $\lb(M^4,\tilde{g},J\rb)$,
\be
\omega\lb(X,Y\rb):=\tilde{g}\lb(X,JY\rb),
\ee
is closed, but $\lb(M^4,\tilde{g},J\rb)$ is not Kähler, i.e.\
\be
\nabla^{\tilde{g}}\omega\neq0.
\ee
We now consider the almost contact metric $5$-manifold $\lb(M^4\x\R,g,\xi,\eta,\vphi\rb)$ defined via (\ref{eqn:3}) and (\ref{eqn:4}). To begin with, we compute the connection forms  of the Levi-Civita connection with respect to the adapted frame $\lb(e_1,\ldots,e_4,e_5:=\xi\rb)$:
\ba
\omega^g_{1i}&=0, &
\omega^g_{23}&=-\frac{1}{2}e_4, &
\omega^g_{24}&=-\frac{1}{2}e_3, &
\omega^g_{34}&=\frac{1}{2}e_2, &
\omega^g_{i5}&=0.
\ea
Consequently, the intrinsic torsion of $\lb(M^4\x\R,g,\xi,\eta,\vphi\rb)$ is
\be
\Gamma=\frac{1}{4}\lb( e_3\ox\lb(e_{13}-e_{24}\rb) - e_4\ox\lb(e_{14}+e_{23}\rb) \rb).
\ee
\begin{prop}\label{prop:5}
The almost contact metric manifold $\lb(M^4\x\R,g,\xi,\eta,\vphi\rb)$ has the following properties:
\bite
\item[a)]The almost contact metric manifold is of class $\W_9$.
\item[b)]Both the fundamental form $\Phi$ and the $1$-form $\eta$ are closed and coclosed.
\item[c)]The Riemannian curvature tensor is
\be
\mrm{R}^g=-\frac{1}{4}\lb(e_{24}\ox e_{24}+e_{34}\ox e_{34}-3\, e_{23}\ox e_{23}\rb).
\ee
\item[d)]The Riemannian Ricci tensor is
\be
\Ric^g=-\frac{1}{2}\diag\lb(0,1,1,-1,0\rb).
\ee
\item[e)]The Riemannian holonomy algebra is
\be
\hol^g=\su\lb(2\rb)\subset\so\lb(5\rb).
\ee
\eite
\end{prop}\noindent
As a result of proposition \ref{prop:5},a) and theorem \ref{thm:5}, there exists a unique metric connection $\nabla^c$ with traceless cyclic torsion compatible with the almost contact metric structure. Its torsion tensor
\be
T^c=\frac{1}{4}\lb(\lb(e_{13}-e_{24}\rb)\ox e_3-\lb(e_{14}+e_{23}\rb)\ox e_4\rb)\in\T_{3,7}
\ee
is not parallel with respect to $\nabla^c$. Moreover, we compute
\begin{prop}
The metric connection $\nabla^c$ has the following properties:
\bite
\item[a)]The curvature tensor with respect to $\nabla^c$ is
\be
\mrm{R}^c=-\frac{1}{8}\lb(
e_{24}\ox\lb(e_{13}+e_{24}\rb)
-e_{34}\ox\lb(e_{12}-e_{34}\rb)
+3\,e_{23}\ox\lb(e_{14}-e_{23}\rb) \rb).
\ee
\item[b)]The Ricci tensor with respect to $\nabla^c$ is
\be
\Ric^c=-\frac{1}{4}\diag\lb(0,1,1,-1,0\rb).
\ee
\item[c)]The holonomy algebra of $\nabla^c$ is
\be
\hol^c=\su\lb(2\rb)\subset\un\lb(2\rb).
\ee
\eite
\end{prop}
%
%
\subsection{Examples of class \texorpdfstring{$\W_8$ and $\W_{10}$}{W_8 and W_10}}
The following Riemannian $5$-manifold is taken from the classification \cite{Kow80} of generalized symmetric Riemannian spaces in low dimensions. Let $G$ be the Lie group
\be
\lb\{
\begin{pmatrix}
1&0&0&0&x_1\\
0&1&0&0&x_2\\
x_5&0&1&0&x_3\\
0&-x_5&0&1&x_4\\
0&0&0&0&1
\end{pmatrix}\in\GL\lb(5,\R\rb)\setsep x_1,\ldots,x_5\in\R\rb\}
\ee
equipped with the left-invariant Riemannian metric
\be
g=\frac{1}{4}dx_1^2+\frac{1}{4}dx_2^2+\lb(dx_3-x_5dx_1\rb)^2+\lb(dx_4+x_5dx_2\rb)^2+dx_5^2.
\ee
The vector fields
\ba
&2\lb(\frac{\del}{\del x_1}+x_5\frac{\del}{\del x_3}\rb), &
&2\lb(\frac{\del}{\del x_2}-x_5\frac{\del}{\del x_4}\rb), &
&\frac{\del}{\del x_3}, &
&\frac{\del}{\del x_4}, &
&\frac{\del}{\del x_5} &
\ea
are left-invariant and dual to
\ba
&\frac{1}{2}dx_1, &
&\frac{1}{2}dx_2, &
&dx_3-x_5dx_1, &
&dx_4+x_5dx_2, &
&dx_5.
\ea
We discuss two almost contact metric structures on $\lb(G,g\rb)$ in detail.
\subsubsection{Class \texorpdfstring{$\W_8$}{W_8}}\label{sec:W8}
There exists an almost contact metric structure $\lb(\xi,\eta,\vphi\rb)$ on $\lb(G,g\rb)$ such that $\lb(e_1,e_2,e_3,e_4,e_5\rb)$,
\ba
e_1&:=2\lb(\frac{\del}{\del x_1}+x_5\frac{\del}{\del x_3}\rb),&
e_2&:=-2\lb(\frac{\del}{\del x_2}-x_5\frac{\del}{\del x_4}\rb),&&\\
e_3&:=\frac{\del}{\del x_3},&
e_4&:=\frac{\del}{\del x_4},&
e_5&:=\frac{\del}{\del x_5},
\ea
is an adapted frame of $\lb(G,g,\xi,\eta,\vphi\rb)$. The non-zero connection forms of the Levi-Civita connection $\nabla^g$ with respect to $\lb(e_1,e_2,e_3,e_4,e_5\rb)$ are
\ba
\omega_{13}^g&=\omega_{24}^g=e_5, &
\omega_{15}^g&=e_3, &
\omega_{25}^g&=e_4, &
\omega_{35}^g&=e_1, &
\omega_{45}^g&=e_2.
\ea
Hence,
\be
\Gamma=e_1\ox e_{35}+e_2\ox e_{45}+e_3\ox e_{15}+e_4\ox e_{25}
\ee
is the intrinsic torsion of $\lb(G,g,\xi,\eta,\vphi\rb)$. Moreover, we have
\begin{prop}\label{prop:6}
The almost contact metric manifold $\lb(G,g,\xi,\eta,\vphi\rb)$ has the following properties:
\bite
\item[a)]The almost contact metric manifold is of class $\W_8$.
\item[b)]The Nijenhuis tensor $N$ vanishes.
\item[c)]The fundamental form $\Phi$ and the $1$-form $\eta$ satisfy
\ba
d\Phi\wedge\Phi&=0,&
\delta\Phi&=0,&
d\eta&=0, &
\delta\eta&=0.
\ea
\item[d)]The Riemannian curvature tensor is
\ba
\mrm{R}^g=
&e_{12}\ox e_{34}+e_{34}\ox e_{12}-e_{14}\ox e_{23}-e_{23}\ox e_{14}-e_{13}\ox e_{13}\\
&-e_{24}\ox e_{24}+3\,e_{15}\ox e_{15}+3\,e_{25}\ox e_{25}-e_{35}\ox e_{35}-e_{45}\ox e_{45}.
\ea
\item[e)]The Riemannian Ricci tensor is
\be
\Ric^g=\diag\lb(-2,-2,2,2,-4\rb).
\ee
\item[f)]The Riemannian holonomy algebra is
\be
\hol^g=\so\lb(5\rb).
\ee
\eite
\end{prop}\noindent
Theorem \ref{thm:5} now implies that there exists a unique metric connection $\nabla^c$ with traceless cyclic torsion preserving the almost contact metric structure $\lb(\xi,\eta,\vphi\rb)$. Its torsion tensor $T^c$ is given by
\be
T^c=e_{35}\ox e_1+e_{45}\ox e_2+e_{15}\ox e_3+e_{25}\ox e_4\in\T_{3,6}
\ee
Consequently, the non-zero connection forms of $\nabla^c$ are
\ba
\omega^c_{13}=\omega^c_{24}=e_5.
\ea
This proves
\begin{prop}\label{prop:8}
The metric connection $\nabla^c$ has the following properties:
\bite
\item[a)]The curvature tensor with respect to $\nabla^c$ vanishes.
\item[b)]The torsion tensor of $\nabla^c$ is parallel with respect to $\nabla^c$.
\eite
\end{prop}
\subsubsection{Class \texorpdfstring{$\W_{10}$}{W_{10}}}
Let $\lb(\xi,\eta,\vphi\rb)$ be the almost contact metric structure on $\lb(G,g\rb)$ such that $\lb(e_1,e_2,e_3,e_4,e_5\rb)$ defined by
\ba
e_1&:=2\lb(\frac{\del}{\del x_1}+x_5\frac{\del}{\del x_3}\rb),&
e_2&:=\frac{\del}{\del x_3},&&\\
e_3&:=-2\lb(\frac{\del}{\del x_2}-x_5\frac{\del}{\del x_4}\rb),&
e_4&:=\frac{\del}{\del x_4},&
e_5&:=\frac{\del}{\del x_5}
\ea
is an adapted frame of $\lb(G,g,\xi,\eta,\vphi\rb)$. With the results of section \ref{sec:W8}, we immediately have
\ba
\omega_{12}^g&=\omega_{34}^g=e_5, &
\omega_{15}^g&=e_2, &
\omega_{25}^g&=e_1, &
\omega_{35}^g&=e_4, &
\omega_{45}^g&=e_3
\ea
and
\be
\omega_{13}^g=\omega_{14}^g=\omega_{23}^g=\omega_{24}^g=0.
\ee
Consequently,
\be
\Gamma=e_1\ox e_{25}+e_2\ox e_{15}+e_3\ox e_{45}+e_4\ox e_{35}.
\ee
\begin{prop}
The almost contact metric manifold $\lb(G,g,\xi,\eta,\vphi\rb)$ is of class $\W_{10}$. Moreover, both the fundamental form $\Phi$ and the $1$-form $\eta$ are closed and coclosed.
\end{prop}
It is easy to verify that the Riemannian curvature tensor and the Riemannian Ricci tensor now read as follows:
\ba
\mrm{R}^g=
&e_{13}\ox e_{24}+e_{24}\ox e_{13}+e_{14}\ox e_{23}+e_{23}\ox e_{14}-e_{12}\ox e_{12}\\
&-e_{34}\ox e_{34}+3\,e_{15}\ox e_{15}-e_{25}\ox e_{25}+3\,e_{35}\ox e_{35}-e_{45}\ox e_{45},\\
\Ric^g=&\diag\lb(-2,2,-2,2,-4\rb).
\ea
By applying theorem \ref{thm:5}, we deduce that $\lb(G,g,\xi,\eta,\vphi\rb)$ admits a unique compatible connection $\nabla^c$ with traceless cyclic torsion. The torsion tensor of $\nabla^c$ is
\be
T^c=e_{25}\ox e_1+e_{15}\ox e_2+e_{45}\ox e_3+e_{35}\ox e_4\in\T_{3,8}.
\ee
Moreover, we compute the non-zero connection forms of $\nabla^c$:
\ba
\omega^c_{12}=\omega^c_{34}=e_5.
\ea
Therefore, proposition \ref{prop:8} is also valid for the connection $\nabla^c$ discussed here.
%
%
%
%

%
%
%
\end{document}